\documentclass[12pt]{article}
\usepackage [english]{babel}
\usepackage{ae,longtable,booktabs,array,subfigure,color,graphicx}
\usepackage{pstricks,pst-node,pst-math,pst-plot,pstricks-add}
\usepackage{times}
\usepackage{url}
\usepackage{amsthm,amsmath,amssymb}
\usepackage{bm,bbm}
\usepackage{paralist}
\usepackage{ifthen}
\allowdisplaybreaks 

\newtheorem{theorem}{Theorem}[section]
\newtheorem{lemma}[theorem]{Lemma}
\newtheorem{cor}[theorem]{Corollary}
\newtheorem{prop}[theorem]{Proposition}
\theoremstyle{definition}

\theoremstyle{remark}
\newtheorem{remark}[theorem]{Remark}

\def\E{\mathbb{E}}

\def\N{\mathbb{N}}

\def\Q{\mathbb{Q}}
\def\R{\mathbb{R}}

\def\T{\mathbb{T}}

\def\cA{\mathcal{A}}

\def\cC{\mathcal{C}}
\def\cF{\mathcal{F}}

\def\cH{\mathcal{H}}

\def\cK{\mathcal{K}}

\newcommand{\nn}{\nonumber}

\usepackage[colorlinks=true, linkcolor=black, citecolor=black, bookmarks]{hyperref}

\newboolean{blackwhite}
\setboolean{blackwhite}{false}
\begin{document}
\author{Julia H\"orrmann, Daniel Hug, Michael Klatt and Klaus Mecke}
\date{\today}

\title{Minkowski tensor density formulas\\ for Boolean models}

\maketitle

\begin{abstract}
A stationary Boolean model is the union set of random compact particles which are attached to the points of a stationary Poisson point process. For a stationary Boolean model with convex grains we consider a recently developed collection of shape descriptors, the so called Minkowski tensors.
By combining spatial and probabilistic averaging we define Minkowski tensor densities of a Boolean model. These densities are global characteristics of the union set which can be estimated from observations. In contrast local characteristics like the mean Minkowski tensor of a single random particle cannot be observed directly, since the particles overlap.
We relate the global to the local properties by density formulas for the Minkowski tensors. These density formulas generalize the well known formulas for intrinsic volume densities and are obtained by applying results from translative integral geometry.
For an isotropic Boolean model we observe that the Minkowski tensor densities are proportional to the intrinsic volume densities, whereas for a non-isotropic Boolean model this is usually not the case. Our results support the idea that the degree of anisotropy of a Boolean model may be expressed in terms of the Minkowski tensor densities. Furthermore we observe that for smooth grains the mean curvature radius function of a particle can be reconstructed from the Minkowski tensor densities. In a simulation study we determine numerically Minkowski tensor densities
 for non-isotropic  Boolean models based on ellipses and on rectangles in
 two dimensions and find excellent agreement with the derived analytic
 density formulas. The tensor densities can be used to characterize the
 orientational distribution of the grains and to estimate model parameters
 for non-isotropic distributions.
 In particular, the numerically determined values for the density of the Euler
 characteristic allow the  estimation of certain mixed functionals of the
 grains.


\vspace*{0.1cm}
\end{abstract}
\renewcommand{\thefootnote}{{}}

\footnote{
The authors acknowledge support by the German research foundation (DFG) through
the research group `Geometry and Physics of Spatial Random Systems' under grants
ME1361/12-1, HU1874/2-1 and HU1847/3-1.}

\renewcommand{\thefootnote}{\fnsymbol{footnote}}

\section{Introduction}
The Boolean model appeared early in applied probability,
 typically in  attempts to describe random geometrical
 structures of physics and materials science by overlapping spherical grains.
 It was Matheron who created the general theory of the Boolean
 model and studied its basic properties \cite{Matheron.1975}.
 Let $\{\xi_i:i\in \mathbb{N}\}$ be a stationary Poisson point process in
 $\mathbb{R}^n$ with intensity $\gamma>0$, and let $K_1,K_2,\ldots$ be independent,
 identically distributed random compact sets with distribution
 $\mathbb{Q}$, which are independent of the point process $\{\xi_i: i\in \N\}$.
 Then, under the integrability assumption \eqref{IntCondBM} on
 $\mathbb{Q}$, the union of the translated grains
 \[
 Z:= \bigcup\limits_{i=1}^\infty (K_i+\xi_i)
 \]
 is a random closed set, which is called the stationary Boolean model with
 intensity $\gamma$ and grain distribution $\mathbb{Q}$; see \cite{Stoyan.2005}
 for a review  of recent developments in this context.

The Boolean model is a popular model in materials science and the physics
 of heterogeneous media \cite{Torquato.2002} relating shape to
 physical properties.   Many  porous materials  are built up
 by the successive addition of inclusions (grains, pores or cracks) within
 a background phase \cite{Garboczi.1995,Arns.2003}; such materials can be
 modeled by a Boolean process.
 It has been applied, in particular, on  foamed materials, ceramic powders
 \cite{Roberts.2000}, wood composites \cite{Wang.1998}, sedimentary rock
 \cite{Schwartz.1993,Martys.1994}, fractured
 materials or hydrating cement-based materials \cite{Garboczi.2011}.
 Depending on the specific application, either the pore space or the solid
 phase of a material may be described as a Boolean model $Z$.
 For example, the pore space of bread \cite{Bindrich.1991} was modeled by
 the Boolean model, whereas for sintered ceramic composites it is the
 solid phase which is described as a Boolean model; see
 \cite{Roberts.2000}.
 In particular, for the reconstruction of two-phase materials
 the Boolean model is successful,
 which finally allows excellent predictions of  the shape dependence of
 thermodynamic quantities \cite{Koenig.2004} and transport
 properties  \cite{Arns.2003} in
 porous media.

Various other physical phenomena can be described and studied by the Boolean
 model including percolation \cite{Mecke.1991,Mecke.2002b} and elasticity
 \cite{Arns.2009}. Many attempts have been made to predict mechanical
 properties from structural features; some of these are based on Boolean models
 and other random set models \cite{Torquato.2002,Jeulin.2005}.
 In practice, measurements taken on samples of a random structure are used
 to adjust the Boolean model to the given structure. Then one tries to draw
 reasonable conclusions from the properties of the properly adjusted
 Boolean model on physical properties of the real structure.
 Obviously it is crucial to take measurements which capture the significant
 geometric properties of the real structure and have as little redundancy as possible. Here the scalar-valued
 intrinsic volume densities have already shown to be a useful choice \cite{Arns.2003,Arns.2009}. On the foundational side, the importance of the intrinsic volumes  $V_0,\ldots,V_n$ is expressed by Hadwiger's  \cite{Hadwiger.1957} famous characterization theorem, which states that  the intrinsic volumes are a basis of the space of real-valued  continuous, additive and motion invariant functionals on the space of convex bodies $\cK$. As a consequence
 of the motion invariance, the intrinsic volumes reach their limits when it comes to the proper characterization of non-isotropic structures. Therefore one is interested in finding functionals which are
 sensitive to anisotropy and have as little redundancy as possible. A set of functionals which serves these purposes
 has been found and explored in recent years. All this has been done in the more general context of convex geometric analysis where the classification and characterization of
 additive functionals (valuations) on the space of convex  bodies $\mathcal{K}$ enjoying specific properties is a highly active field of research. Generalizations of Hadwiger's result, which concerns
  scalar-valued functionals, to vector-valued valuations which are motion covariant have already been found in the
  early '70s by Hadwiger and Schneider \cite{Hadwiger.1971b, Schneider.1972b, Schneider.1972c}. More recently,  tensor-valued valuations of higher rank have come into focus and it immediately turned out that in this case a basis cannot be determined that easily. The current mathematical study of tensor valuations has been initiated by McMullen \cite{McMullen.1997}.
  For $K\in\mathcal{K}$ and integers $r,s\geq 0$ and $0 \leq j\leq n-1$, the Minkowski tensors are defined by
\begin{equation}\label{zita}
\Phi_j^{r,s}(K):= \frac{1}{r!s!}\frac{\omega_{n-j}}{\omega_{n-j+s}} \int\limits_{\Sigma} x^r u^s \Lambda_j(K;d(x,u))
\end{equation}
and
\begin{equation}\label{zitb}
\Phi_n^{r,0}(K):= \frac{1}{r!}\int\limits_{\mathbb{R}^n}x^r \Lambda_n(K;dx),
\end{equation}
where we use the notation introduced in Section \ref{SectionPrel}.
In  \cite{McMullen.1997} McMullen conjectured that the  basic tensor valuations $Q^{m}\Phi_j^{r,s}, r,s,m\in \N_0$ with $r+s+2m=p$, span the space of continuous, additive and isometry covariant $\T^p$-valued functionals, for every $p\in \N_0$. Furthermore, it was already observed by McMullen that the basic tensor valuations satisfy linear dependencies and therefore do not form a basis of the vector space they span.
McMullen's conjecture was almost immediately confirmed by Alesker \cite{Alesker.1999, Alesker.1999b}. Later it was shown how a basis can be constructed and the dimension of the corresponding vector spaces was determined; see  \cite{Hug.2008b}.
More information on the mathematical and physical background of the Minkowski tensors can be found in \cite{Schneider.a,Hug.2008,Hug.2008b,SchroderTurk.2011,Mickel.2012}.
In particular, in \cite{SchroderTurk.2011,Mickel.2012} Minkowski tensors have already proved to be useful shape descriptors for anisotropic structures. Of the many characterization theorems for valuations with values in some  abelian group $G$ and related to the present work, we only mention \cite{Schneider.1978, Ludwig.2002, Ludwig.2003,  Ludwig.2013, Schneider.2012, Hug.Schneider.2013} which are concerned with  characterizations of curvature measures, moment vectors, moment matrices, covariance matrices and local tensor valuations.

A nice visualisation of a Boolean model can be obtained by a  dewetting
 process
 of thin liquid films.  Films rupture at random distributed defects and
 form holes, which grow in time, until the fluid material is pushed
 completely in thin filaments, which finally break up in droplets
 \cite{Jacobs.1998,Mantz.2008}. The time dependence of the experimentally
 measured Minkowski functionals of the film regions can be compared with
 analytic Minkowski density formulas for a Boolean model, which shows a
 good agreement.
 Other applications of a Boolean model on characterising spatial patterns
 are possible in the analysis of inhomogeneous
 distributions of galaxies \cite{Buchert.1994,Kerscher.2001}, the measurement of
 biometrical data  \cite{Mecke.2005}, or in the estimation of percolation
 thresholds \cite{Mecke.1991,Mecke.2002b}.
 While in most of the applications, the Boolean model was  isotropic (and
 stationary), we focus here the non-isotropic case  and
 analyze its distributional properties by tensorial quantities, the Minkowski
 tensors.
 Non-isotropic (and non-stationary) Boolean models should be exploited
 further,
 since many   applications are possible on composite and porous media but
 also on complex fluids such as colloidal dispersions, which show
 qualitatively rich phase diagrams and spatial structures. Minkowski
 tensors have already been used, for instance, to derive a density functional theory
 for non-spherical particles \cite{HansenGoos.2009,HansenGoos.2010}. Furthermore, local stereological estimators of the Minkowski tensors have been developed based on rotational integral formulas \cite{AuneauCognacq.2013,Jensen.2013}.
 Finally, we emphasize that in two and three dimensions, algorithms and free software for the computation of the Minkowski tensors are available; see \cite{SCHRODERTURK.2010,Mickel.2012}.

The structure and scope of the article are as follows. Section \ref{SectionPrel} introduces the notation and necessary mathematical background. In Section \ref{SectionTransl}, we recall in Theorem \ref{TransIntFormSupMeas} a translative integral formula for support measures and we deduce as an application a corresponding formula for Minkowski tensors (Theorem \ref{translativeMink}). Thereby we introduce mixed Minkowski tensors. In Section \ref{SectionDefDensities}, densities of a particle process and densities of a standard random set are defined for the translation invariant Minkowski tensors. Section \ref{SectionMeanValDensFormulas} contains the main results, namely a formula for the expected surface and volume tensor of a Boolean model observed in a window (Theorem \ref{ExpOfSectWithWindow}) and a corresponding result for the other Minkowski tensors (Theorem \ref{ExpectFormulasGeneral}).
To give a first impression we state the result in two dimensions in the following corollary using the notation which is introduced in the preliminaries and the subsequent sections.
\begin{cor}\label{ExpOfSectWithWindowDimTwo}
Let $Z$ be a stationary Boolean model in $\mathbb{R}^2$ with convex grains, let $W\in\mathcal{K}$ and $r,s\in\N_0$.
Then
\begin{align*}
\mathbb{E}\left[\Phi_0^{r,s}(Z \cap W)\right]
& = \Phi_0^{r,s}(W)\left(1-\mathrm{e}^{-\overline{V}_2(X)}\right)\\
&\quad +\mathrm{e}^{-
\overline{V}_2(X)}
\Big[ \ ^0\overline{\Phi}^{\,r,s}_{1,1}(W,X)
\mathbf{1}\{{s}\in 2\mathbb{N}_0\}\, \frac{2}{s! \omega_{s+1}}\, Q^{\frac{s}{2}} \Phi_2^{r,0}(W)\\
&\quad\quad\times
\left(\overline{V}_0(X)- \frac{1}{2}  \ ^0\overline{V}_{1,1}(X,X) \right)\Big];\\
\mathbb{E}\left[\Phi_{1}^{r,s}(Z\cap W)\right] &= \Phi_{1}^{r,s}(W)\left(1-\mathrm{e}^{-\overline{V}_2(X)}\right)+\Phi_2^{r,0}(W) \overline{\Phi}_{1}^{\,0,s}(X) \mathrm{e}^{-\overline{V}_2(X)};\\
\mathbb{E}[ \Phi_2^{r,0}(Z\cap W)]&=
\Phi_2^{r,0}(W)\left( 1-\mathrm{e}^{-\overline{V}_2(X)}\right).
\end{align*}
\end{cor}
If a Boolean model is observed, the quantities on the left-hand side of Corollary \ref{ExpOfSectWithWindowDimTwo} can be measured using the above mentioned software. The right-hand side involves quantities depending either only on the particle process $X$ associated with the Boolean model, and therefore only on the model parameters, or merely on the observation window $W$ or (and this case occurs only in the first equation) on both, $W$ and $X$. The last two equations  show that the quantities on the left-hand side for $r>0$ do not contain additional information about the Boolean model compared to the case $r=0$. This is not the case in the first equation   though there information is hidden in the mixed functionals and therefore difficult to extract. An easy interpretation exists for the information contained in the quantities on the left-hand side of the second equation for $s>0$, namely in terms of Fourier coefficients as explained in Subsection \ref{BMSmoothGrains}.

Then, Section \ref{SectionMeanValDensFormulas} contains formulas for the densities of the Boolean model for all translation invariant Minkowski tensors (Corollary \ref{DensityArbDim}), as well as a collection of density formulas in two and three dimensions (Corollary \ref{DensityDimTwoThree}). Furthermore, we show in a second part of Section \ref{SectionMeanValDensFormulas} that densities of the Minkowski tensors for isotropic standard random sets are just multiples of the densities of the intrinsic volumes. In the first part of Section \ref{SectionExamples} we discuss for a non-isotropic parametric planar Boolean model which information about the model parameters is contained in the Minkowski tensor densities. In the second part we show for a planar Boolean model with smooth grains that the expected curvature radius of the typical grain multiplied with the intensity can be expressed almost everywhere in terms of the surface tensor densities. In Section \ref{SectionSimulations} we carry out a simulation study for the parametric Boolean model from the previous section. We compare the analytical formulas for the surface tensor density and for the Euler characteristic to measurements on simulated data and we estimate the model parameters from the measurements of the volume fraction and the surface tensor density.

\section{Preliminaries}\label{SectionPrel}
We denote by $\mathbb{T}^p$ the vector space of symmetric tensors of rank $p$ over $\mathbb{R}^n$.
 We use the scalar product to identify $\mathbb{R}^n$ with its dual space; then $\mathbb{T}^p$ can be viewed as the vector space of symmetric $p$-linear functionals on $\mathbb{R}^n$. If we choose a basis $\{e_1,\ldots,e_n\}$ of $\mathbb{R}^n$, a tensor $T\in\mathbb{T}^p$ is uniquely determined by the $\binom{n+p-1}{p}$ values $ T(e_{i_1},\ldots,e_{i_p})$, $ 1\leq i_1\leq\ldots\leq i_p\leq n$.
Therefore, we can identify $\mathbb{T}^p$ with a ${\binom{n+p-1}{p}}$-dimensional Euclidean space, a fact which will be often useful. We define the norm $|\cdot|_{\infty}$ on $\mathbb{T}^p$ as the maximum norm on such a ${\binom{n+p-1}{p}}$-dimensional Euclidean space.
 The symmetric tensor product of symmetric tensors $a, b$ is denoted by $ab$, and $x^r$ is the $r$-fold symmetric tensor product of $x\in\mathbb{R}^n$.
The metric tensor $Q\in \T^2$ is defined by $Q(x,y)=\langle x,y\rangle$, for $x,y \in \R^n$.
By $\mathcal{K}$ we denote the family of nonempty, compact, convex subsets (convex bodies) of $\mathbb{R}^n$.
The system of nonempty, compact subsets of $\mathbb{R}^n$ is denoted by $\mathcal{C}$.
Let $A$ be a subset of $\mathbb{R}^n$. Then $\text{int\,}A, \partial A, \text{relint\,}A$ are, respectively, the interior, the boundary and the relative interior of $A$.
Let $\langle \cdot\,, \cdot \rangle$ be the scalar product and $\|\cdot\|$ the norm in $\mathbb{R}^n$.

A measure or signed measure on a topological space $E$ will always be defined on the $\sigma$-algebra $\mathcal{B}(E)$ of Borel sets. Lebesgue measure on $\mathbb{R}^n$ is denoted by $\lambda$. The $k$-dimensional Hausdorff measure is denoted by $\mathcal{H}^k$. By $\cH^k\llcorner A$ we denote the restriction of $\cH^k$ to a subset $A$.
We denote the unit ball by $B^n$, the unit sphere by $S^{n-1}$ and the unit cube by $C^n:= [0,1]^n$. Furthermore we shall need the `upper right boundary'
$$\partial^+ C^n :=\{(x_1,\ldots,x_n)\in \mathbb{R}^n: \max\limits_{1\leq i\leq n} x_i=1\}
$$
of the unit cube.
The volume of the unit ball $B^n$ is denoted by $\kappa_n:= \lambda(B^n)= {\pi^{{n}/{2}}}/{\Gamma(1+{n}/{2})}$
and the surface area of the unit sphere $S^{n-1}$ is given by $\omega_n:=n\kappa_n$.
The group of proper rotations is denoted by $SO_n$ and it is equipped with its standard topology. The unique normalized Haar measure on $SO_n$ is denoted by $\nu$.
 For $x \in \mathbb{R}^n$, let $p(K,x)$ denote the metric projection of $x$ to $K$. For $x\in K$ we define the normal cone of $K$ at $x$ by
$N(K,x):= \{u\in\mathbb{R}^n: p(K,x+u)=x\}$,
and for nonempty, convex $F \subset K$ let $N(K,F):= N(K,x)$, where $x\in \text{relint\,} F$.
For $x\notin K$ put $u(K,x):= (x-p(K,x))/\|x-p(K,x)\|$.
We need the support measures (generalized curvature measures) $\Lambda_0(K;\cdot), \ldots, \Lambda_{n-1}(K;\cdot) $ of a convex body $K \in \mathcal{K}$, which are defined by a local Steiner formula.
 Namely, for any $\epsilon >0$ and Borel set $\eta \subset \Sigma:= \mathbb{R}^n \times S^{n-1}$, the $n$-dimensional Hausdorff measure (volume) of the local parallel set
$
M_\epsilon(K,\eta):=\{x\in(K+\epsilon B^n)\setminus K: (p(K,x),u(K,x))\in \eta\}
$
is a polynomial in $\epsilon$, that is,
\[
\mathcal{H}^n(M_\epsilon(K,\eta))=\sum\limits_{k=0}^{n-1} \epsilon^{n-k}\kappa_{n-k} \Lambda_k(K;\eta);
\]
see \cite{Schneider.1993,Schneider.2008d} for further information. The support measures are related to the intrinsic volumes which are defined by $V_i(K) := \Lambda_i(K;\Sigma)$, for $i=0,\ldots,n-1$, and by $V_n(K)=\lambda(K)$. In addition, we define $\Lambda_n(K;\cdot)$ as the restriction of $\mathcal{H}^n$ to $K$. Once the support
measures are available, the Minkowski tensors can be defined as in \eqref{zita}, \eqref{zitb} in a straightforward way by integration of tensor-valued functions.

The convex ring $\mathcal{R}$ consists of all finite unions of convex bodies and its elements are called polyconvex sets. By additivity, the support measures and hence also the Minkowski tensors can be extended to $\mathcal{R}$.
The extended convex ring $\mathcal{S}$ is the system of sets whose intersection with any compact convex set belongs to the convex ring and its elements are called locally polyconvex sets.
For $p\in\mathbb{N}$ a function $\varphi: \mathcal{R}\rightarrow \mathbb{T}^p$ is conditionally bounded if, for $K\in\mathcal{K}$, the function is bounded on the set $\{L \in \mathcal{K}: L \subset K\}$ with respect to the norm $|\cdot|_\infty$.
We define by $c: \mathcal{C}\rightarrow \mathbb{R}^n$
the mapping that associates with each $C\in\mathcal{C}$ the center of the (uniquely determined) smallest ball containing $C$. The mapping $c$ is continuous with respect to the Hausdorff metric; see \cite[Lem. 4.1.1]{Schneider.2008d}.
Furthermore, we define the grain space $\mathcal{C}_0:= \{C\in\mathcal{C}: c(C)=0\}$
and correspondingly $\mathcal{K}_0:=\mathcal{C}_0\cap\mathcal{K}$ and $\mathcal{R}_0:=\mathcal{C}_0\cap\mathcal{R}$.

For $R\in\mathcal{R}$, we define $N(R):=\min \{ m \in\mathbb{N} : R=\bigcup_{i=1}^m K_i \text{ with } K_i\in\mathcal{K}\}$
and $N(\emptyset)=0.$ The function $N:\mathcal{R}\cup \{\emptyset\}:\rightarrow \mathbb{N}_0$ is measurable, compare \cite[Lem. 4.3.1]{Schneider.2008d}.
In the following we shall need a more theoretical viewpoint of the Boolean model which is used for example in \cite{Schneider.2008d}. We assume that the grain distribution $\Q$ is concentrated on $\cC_0$. Let $Z_0$ be the typical grain of $Z$, that is, a random compact set with distribution $\Q$. Then, under the integrability condition
\begin{equation}\label{IntCondBM}
\E\left[ \lambda(Z_0+B^n)\right]<\infty,
\end{equation}
there exists a unique Poisson point process $X$ in $\cK$ with intensity measure
$$
\Theta(\cA):=\gamma \int\limits_{\R^n}\int\limits_{\cK_0}\mathbf{1}_{\cA}(K+x)\Q(dK)dx,\quad \cA\in \mathcal{B}(\cK).
$$
The Boolean model $Z$ is the random closed set which is obtained as the union of the particles of $X$, that is,
\[
Z=\bigcup\limits_{K\in X} K.
\]
More information on the Boolean model can be found in  \cite{Stoyan.1995,Molchanov.1997} and \cite{Schneider.2008d}.
%
%
\section{Translative Integral Formulas}\label{SectionTransl}
In the following, we shall need an iterated translative integral formula which has been proved in the setting of sets with positive reach \cite{Rataj.1997, Hug.Rataj+} and in the framework of relative support measures in \cite{Hug.1999} (partly based on \cite{Kiderlen.Weil.1999, Hug.Last.2000}). The formula stated in Theorem \ref{TransIntFormSupMeas} is a special case of each of these more general versions.
For the statement of the theorem, we need the notion of a determinant of subspaces.
Let $L_1,\ldots,L_k \subset \mathbb{R}^n$ be linear subspaces with $\dim L_1+\ldots +\dim L_k=:m \leq n$. Then we choose an orthonormal basis in each subspace $L_j$ and define $\det(L_1,\ldots,L_k)$
as the $m$-dimensional volume of the parallelepiped which is spanned by the union of these orthonormal bases. On the other hand, if $\dim L_1+\ldots +\dim L_k \geq (k-1)n$, we define
\[ [L_1,\ldots,L_k]:= \det(L_1^\bot,\ldots,L_k^\bot).
\]
Moreover, if $A_1,\ldots,A_k$ are non-empty convex sets with
$\dim A_1+\ldots +\dim A_k \geq (k-1) n$
 and $L(A_i)$ denotes the linear subspace which is parallel to $\text{aff\,}A_i$, then we define
\[
[A_1,\ldots,A_k]:=[L(A_1),\ldots,L(A_k)].
\]
For a polytope $P$ and $0\leq k\leq n$ we denote by $\cF_k(P)$ the set of all $k$-faces.

\begin{theorem}\label{TransIntFormSupMeas}
Let $K_1,\ldots,K_k\in\mathcal{K}$, $j\in\{0,\ldots,n-1\}$, and $k\in\mathbb{N}$. Further, let $f:(\mathbb{R}^n)^{k}\times S^{n-1}\rightarrow \mathbb{R}$ be a nonnegative Borel measurable function. Then there exist (uniquely determined) Borel measures $\Lambda^{(j)}_{m_1,\ldots,m_k}(K_1,\ldots,K_k;\cdot)$ on $(\mathbb{R}^n)^{k}\times S^{n-1}$, for $m_1,\ldots,m_k\in\{j,\ldots,n\}$ with $m_1+\ldots+m_k=(k-1)n+j$, such that
\begin{align*}
& \int\limits_{\mathbb{R}^n}\ldots\int\limits_{\mathbb{R}^n}\int\limits_{\mathbb{R}^n\times S^{n-1}}
f(z,z-x_2,\ldots,z-x_k,u)\\
&\times \Lambda_j(K_1\cap(K_2+x_2)\cap\ldots\cap(K_k+x_k);d(z,u)) dx_2 \ldots
dx_k\\
& \quad = \sum\int\limits_{(\mathbb{R}^n)^{k}\times S^{n-1}}f(x_1,\ldots,x_k,u)\Lambda^{(j)}_{m_1,\ldots,m_k}(K_1,\ldots,K_k;d(x_1,\ldots,x_k,u)),
\end{align*}
where the summation extends over all $m_1,\ldots,m_k\in \{ j,\ldots,n \}$ such that $m_1+\ldots + m_k= (k-1)n+j$.
Let $A_i\subset \mathbb{R}^n$, $i \in\{1,\ldots,k\}$, $C\subset S^{n-1}$, $D^\prime\subset (\mathbb{R}^n)^{k-1}\times S^{n-1}$ and $D\subset(\mathbb{R}^n)^{k}\times S^{n-1}$ be Borel sets. Then the following is true:
\begin{itemize}
\item [{\rm(i)}]
$\Lambda^{(j)}_{m_1,\ldots,m_k}(K_1,\ldots,K_k;A_1\times\ldots\times A_k\times C)$ is symmetric with respect to permutations of $\{1,\ldots,k\}$;
\item [{\rm(ii)}] $\Lambda^{(j)}_j(K_1;A_1\times C)=\Lambda_j(K_1;A_1\times C)$ and
$$\Lambda^{(j)}_{n,m_2,\ldots,m_k}(K_1,\ldots,K_k;A_1\times D^\prime)\\
= \mathcal{H}^n(K_1\cap A_1)\Lambda^{(j)}_{m_2,\ldots,m_k}(K_2,\ldots,K_k;D^\prime);
$$
\item [{\rm(iii)}]
$\Lambda^{(j)}_{m_1,\ldots,m_k}(K_1,\ldots,K_k;\cdot)$ is a finite nonnegative Borel measure on $(\mathbb{R}^n)^{k}\times S^{n-1}$ which is supported by $S_1\times\ldots\times S_k\times S^{n-1}$, where $S_i=K_i$ if $m_i=n$, and $S_i = \partial K_i$ otherwise;
\item [{\rm(iv)}]
$\Lambda^{(j)}_{m_1,\ldots,m_k}(K_1,\ldots,K_k;A_1\times\ldots\times A_k\times C)$ is positively homogeneous of degree $m_i$ with respect to $(K_i,A_i)$;
\item [{\rm(v)}]
if $K_1,\ldots,K_k$ are polytopes, then
\begin{align*}
 &\Lambda^{(j)}_{m_1,\ldots,m_k}(K_1,\ldots,K_k;A_1\times\ldots\times A_k\times C)\\
 &\qquad = \sum\limits_{F_1\in\mathcal{F}_{m_1}(K_1)}\ldots \sum\limits_{F_k\in\mathcal{F}_{m_k}(K_k)}\frac{\mathcal{H}^{n-1-j}\left(\left(\sum\limits_{i=1}^k N(K_i,F_i)\right)\cap C\right)}{\omega_{n-j}}\\
 & \qquad\qquad\times
 [F_1,\ldots,F_k](\mathcal{H}^{m_1}\llcorner F_1)(A_1)\cdots (\mathcal{H}^{m_k}\llcorner F_k)(A_k);
\end{align*}

\item [{\rm(vi)}]
the map $(K_1,\ldots,K_k)\mapsto \Lambda^{(j)}_{m_1,\ldots,m_k}(K_1,\ldots,K_k;\cdot)$ from $(\mathcal{K})^k$ into the space of finite Borel measures on $(\mathbb{R}^n)^{k+1}$ is weakly continuous;
\item [{\rm(vii)}]
the map $(K_1,\ldots,K_k)\mapsto \Lambda^{(j)}_{m_1,\ldots,m_k}(K_1,\ldots,K_k;D)$ defined on $(\mathcal{K})^k$ is measurable;
\item [{\rm(viii)}]
the map $(K_1,\ldots,K_k)\mapsto\Lambda^{(j)}_{m_1,\ldots,m_k}(K_1,\ldots,K_k;\cdot)$ is additive in each of the first $k$ components;
\item [{\rm(ix)}]
if $(K_1^\prime,\ldots,K_k^\prime)\in (\mathcal{K})^k, \beta_1,\ldots,\beta_k\subset \mathbb{R}^n$ are open sets and $K_i\cap \beta_i = K_i^\prime \cap \beta_i$, for $i=1,\ldots,k$, then
\[
\Lambda^{(j)}_{m_1,\ldots,m_k}(K_1,\ldots,K_k;\cdot)=\Lambda^{(j)}_{m_1,\ldots,m_k}(K^\prime_1,\ldots,K^\prime_k;\cdot)
\]
on Borel subsets of $\beta_1\times\ldots\times\beta_k\times S^{n-1}$;
\item [{\rm(x)}]
$\Lambda^{(j)}_{m_1,\ldots,m_k}(K_1+x_1,\ldots,K_k+x_k;(A_1+x_1)\times\ldots\times (A_k+x_k)\times C)$

$ =\Lambda^{(j)}_{m_1,\ldots,m_k}(K_1,\ldots,K_k;A_1\times\ldots\times A_k\times C)$
for $ x_1,\ldots,x_k\in\mathbb{R}^n$.
\end{itemize}
\end{theorem}
\begin{proof}
For the proof compare \cite[Thm. 3.14]{Hug.1999}, which states a corresponding formula for relative support measures accompanied by the properties of mixed relative support measures. The property ${\rm (v)}$ of the classic support measures can be found in \cite[Cor. 4.10]{Hug.1999}.
Property ${\rm(x)}$ follows for polytopes $P_1,\ldots,P_k$ from property ${\rm (v)}$. Using algebraic induction, the weak continuity and an approximation argument it is obtained for arbitrary convex bodies.
\end{proof}
Now we use Theorem \ref{TransIntFormSupMeas} for the study of the translative integral of a Minkowski tensor.
In the following we shall apply integrals and limits to tensors meaning the application to the real-valued coordinates of a basis representation.
An important role is played by the following mixed tensorial functionals, which will be called {\it mixed Minkowski tensors}.

For $K_1,\ldots,K_k\in\mathcal{K}$, $j\in\{0,\ldots,n-1\}$, $k\in\mathbb{N}$ and $m_1,\ldots,m_k\in\{j,\ldots,n\}$ with $m_1+\ldots+m_k=(k-1)n+j$ let
$$
 ^j\Phi^{r,s}_{m_1,\ldots,m_k}(K_1,\ldots,K_k)
:= \frac{1}{r!s!}\frac{\omega_{n-j}}{\omega_{n-j+s}} \int\limits_{(\mathbb{R}^n)^k\times S^{n-1}} x_1^r u^s \Lambda^{(j)}_{m_1,\ldots,m_k}(K_1,\ldots,K_k;d(x_1,\ldots,x_k,u)).
$$
A special case of the mixed Minkowski tensors are the mixed functionals of translative integral geometry
\[
\ ^j V_{m_1,\ldots,m_k}:=\ ^j \Phi_{m_1,\ldots,m_k}^{0,0}
\]
with $j, k$ and $m_1,\ldots,m_k$ as above. To keep the notation consistent we deviate from the more common notation $V^{(j)}_{m_1,\ldots,m_k}$ which is used in \cite{Schneider.2008d}.
Theorem \ref{TransIntFormSupMeas} leads to the following translative integral formula for Minkowski tensors.
%
%
\begin{theorem}\label{translativeMink}
Let $K_1,\ldots,K_k\in\mathcal{K}$, $j\in\{0,\ldots,n-1\}$, $k,r,s\in\mathbb{N}$. Then
$$
\int\limits_{\mathbb{R}^n}\ldots\int\limits_{\mathbb{R}^n}
\Phi_j^{r,s}(K_1\cap(K_2+x_2)\cap\ldots\cap(K_k+x_k))
dx_2\ldots
dx_k = \sum\ ^j\Phi^{r,s}_{m_1,\ldots,m_k}(K_1,\ldots,K_k),
$$
where the summation extends over all $m_1,\ldots,m_k\in \{ j,\ldots,n \}$ such that $m_1+\ldots + m_k= (k-1)n+j$.

\begin{itemize}
\item [{\rm(i)}]
$^j\Phi^{r,s}_{m_1,\ldots,m_k}(K_1,\ldots,K_k)$ is symmetric with respect to permutations of $\{2,\ldots,k\}$. For $r=0$ it is even symmetric with respect to permutations of $\{1,\ldots,k\}$;

\item [{\rm(ii)}]
$^j\Phi^{r,s}_{n,m_2,\ldots,m_k}(K_1,\ldots,K_k)=\Phi^{r,0}_n(K_1)\ ^j\Phi^{0,s}_{m_2,\ldots,m_k}(K_2,\ldots,K_k)$,

$^j\Phi^{r,s}_{m_1,n,m_3,\ldots,m_k}(K_1,\ldots,K_k)=V_n(K_2)\ ^j\Phi^{r,s}_{m_1,m_3,\ldots,m_k}(K_1,K_3,\ldots,K_k)$, and
$^j \Phi^{r,s}_j(K)=\Phi_j^{r,s}(K)$;

\item [{\rm(iii)}]
$^j\Phi^{r,s}_{m_1,\ldots,m_k}(K_1,\ldots,K_k)$ is positively homogeneous of degree $m_1+r$ with respect to $K_1$ and of degree $m_i$ with respect to $K_i$ for $i\geq 2$;

\item [{\rm(iv)}]
if $K_1,\ldots,K_k$ are polytopes, then
\begin{align*}
&^j\Phi^{r,s}_{m_1,\ldots,m_k}(K_1,\ldots,K_k)\\
&\qquad = \frac{1}{r!s!} \frac{1}{\omega_{n-j+s}}
\sum\limits_{F_1\in\mathcal{F}_{m_1}(K_1)}\ldots \sum\limits_{F_k \in\mathcal{F}_{m_k}(K_k)} \,
\int\limits_{\big(\sum\limits_{i=1}^k N(K_i,F_i)\big)\cap S^{n-1}} u^s \,\mathcal{H}^{n-1-j}(du)\\
& \qquad\qquad \times [F_1,\ldots,F_k]\, \int\limits_{F_1}x_1^r\mathcal{H}^{m_1}(dx_1)\,\mathcal{H}^{m_2}(F_2)\cdots \mathcal{H}^{m_k}(F_k);
\end{align*}

\item [{\rm(v)}]
The map $(K_1,\ldots,K_k)\mapsto\ ^j\Phi^{r,s}_{m_1,\ldots,m_k}(K_1,\ldots,K_k)$ is additive and continuous with respect to the Hausdorff metric in each component.
\end{itemize}
\end{theorem}
\begin{proof}
The formula follows by applying the previous theorem with the special integrand
\[
f: \begin{cases}\begin{array}{rll}
(\mathbb{R}^n)^k\times S^{n-1}& \rightarrow & \mathbb{R},\\
(x_1,\ldots,x_k,u)&\mapsto& \frac{1}{r!s!} \frac{\omega_{n-j}}{\omega_{n-j+s}}x_1^r u^s,
\end{array}\end{cases}
\]
more precisely, the theorem has to be applied to the positive and negative part of a representation of $f$ with respect to a fixed basis. The existence of the translative integral follows since
$K_1\cap (K_2+x_2)\neq \emptyset$ holds if and only if $x_2\in K_1+(-K_2)$ and since the intrinsic volumes are increasing with respect to set inclusion. Hence,
 \begin{align*}
& \int\limits_{\mathbb{R}^n}\ldots\int\limits_{\mathbb{R}^n}
|\Phi_j^{r,s}(K_1\cap(K_2+x_2)\cap\ldots\cap(K_k+x_k))|_\infty
dx_2\ldots
dx_k\\
& \leq \int\limits_{K_1+(-K_2)}\ldots\int\limits_{K_1+(-K_k)}
|\Phi_j^{r,s}(K_1\cap(K_2+x_2)\cap\ldots\cap(K_k+x_k))|_\infty dx_2\ldots
dx_k\\
& \leq \int\limits_{K_1+(-K_2)}\ldots\int\limits_{K_1+(-K_k)}
\frac{1}{r!s!}\frac{\omega_{n-j}}{\omega_{n-j+s}} \left(\max\limits_{x\in K_1} \|x\| \right)^r\,
V_j(K_1) dx_2\ldots
dx_k < \infty.
\end{align*}
The properties ${\rm(i)}$ to ${\rm(v)}$ follow from the corresponding properties of the mixed support measures, compare Theorem \ref{TransIntFormSupMeas}.
\end{proof}
In the special case $j=n-1$, Theorem \ref{translativeMink} reduces to the following corollary, which does not require mixed Minkowski tensors.

\begin{cor}\label{IteratedTranslated}
Let $k\in \mathbb{N}, K_1,\ldots, K_k \in\mathcal{K}, r,s\in\mathbb{N}_0$. Then
\begin{align*}
& \int\limits_{(\mathbb{R}^n)^{k-1}}
\Phi_{n-1}^{r,s}(K_1\cap(K_2+x_2)\cap\ldots\cap(K_k+x_k))
dx_2 \ldots dx_k\\
&  = \Phi_{n-1}^{r,s}(K_1) V_n(K_2)\cdots V_n(K_k)\\
& \quad
 + \sum\limits_{l=2}^{k}\big(\Phi_n^{r,0}(K_1) V_n(K_2)\cdots V_n(K_{l-1})\Phi_{n-1}^{0,s}(K_l) V_n(K_{l+1})\cdots V_n(K_k)\big)
\end{align*}
and
$$
\int\limits_{(\R^n)^{k-1}} \Phi_n^{r,0}(K_1\cap (K_2+x_2)\cap\ldots \cap (K_k+x_k))dx_2\ldots dx_k= \Phi_n^{r,0}(K_1)V_n(K_2)\cdots V_n(K_k).
$$
\end{cor}
\begin{proof}
By Theorem \ref{translativeMink}, we have
\begin{align*}
& \int\limits_{(\mathbb{R}^n)^{k-1}}
\Phi_{n-1}^{r,s}(K_1\cap(K_2+x_2)\cap\ldots\cap(K_k+x_k))
dx_2 \ldots dx_k\\
& = \sum\limits_{\substack{m_1,\ldots,m_k=n-1\\ m_1+\ldots +m_k=nk-1}}^n \ ^j\Phi^{r,s}_{m_1,\ldots,m_k}(K_1,\ldots,K_k)\\
& =\sum\limits_{l=1}^k \ ^j\Phi^{r,s}_{n,\ldots,n,\underbrace{\scriptstyle n-1}_{l\text{th comp.}},n,\ldots,n}(K_1,\ldots,K_k)\\
& =\Phi_{n-1}^{r,s}(K_1) V_n(K_2)\cdots V_n(K_k)\\
& \quad + \sum\limits_{l=2}^{k}\big(\Phi_n^{r,0}(K_1) V_n(K_2)\cdots V_n(K_{l-1})\Phi_{n-1}^{0,s}(K_l) V_n(K_{l+1})\cdots V_n(K_k)\big).
\end{align*}
The second relation follows by an application of Fubini's theorem.
\end{proof}
\section{Densities of Stationary Models}\label{SectionDefDensities}
In this section, we define densities of (mixed) Minkowski tensors for particle processes and random closed sets.

Recall from the end of Section \ref{SectionPrel} the stationary particle process $X$ which is associated with the Boolean model $Z$.
For a real-valued, translation invariant, measurable functional $\varphi:\mathcal{C} \rightarrow\mathbb{R}$, the $\varphi$-density of $X$ is
defined by
\[
\overline{\varphi}(X):=\gamma \int\limits_{\mathcal{C}_0}\varphi\, d\mathbb{Q},
\]
if $\varphi$ is nonnegative or $\mathbb{Q}$-integrable (cf.~\cite[(4.6)]{Schneider.2008d}).

In order to extend this definition to the setting of the (mixed) Minkowski tensors, we provide a lemma.

\begin{lemma}\label{LemPropForDens}
\begin{itemize}
\item [{\rm(i)}] Let $r,s\in\mathbb{N}_0$ and $j\in\{0,\ldots,n\}$. Then the functional $\Phi_j^{r,s}$ on $\mathcal{K}$ has an additive extension to $\mathcal{R}$, which will be denoted by the same symbol. The extension
 \[
 \Phi_j^{r,s}: \mathcal{R}\rightarrow \mathbb{T}^{r+s}
 \]
is measurable and conditionally bounded.
\item [{\rm(ii)}] Let $K_2,\ldots,K_k\in\mathcal{K}$, $r,s \in\mathbb{N}_0$, $j\in\{0,\ldots,n-1\}$, $k\in\mathbb{N}$ and $m_1,\ldots,m_k\in\{j,\ldots,n\}$ with $m_1+\ldots+m_k=(k-1)n+j$. Then the functional
\[
^j\Phi_{m_1,\ldots,m_k}^{r,s}(\cdot,K_2,K_3,\ldots,K_k)
\]
 on $\mathcal{K}$ has an additive extension to $\mathcal{R}$, which will be denoted by the same symbol. The extension
\[
^j\Phi_{m_1,\ldots,m_k}^{r,s}(\cdot,K_2,K_3,\ldots,K_k): \mathcal{R}\rightarrow \mathbb{T}^{r+s}
\]
is measurable and conditionally bounded.
The same holds with respect to the arguments $K_2,\ldots,K_k$.
\end{itemize}
\end{lemma}
\begin{proof}
We can identify $\mathbb{T}^{r+s}$ with a $\binom{n+r+s-1}{r+s}$-dimensional Euclidean space and so results on real-valued functions on $\mathcal{K}$ are still true for tensor-valued functions on $\mathcal{K}$ if applied coordinate-wise.
Both functionals, $\Phi_j^{r,s}$ and $^j\Phi_{m_1,\ldots,m_k}^{r,s}(\cdot,K_2,K_3,\ldots,K_k)$, are additive and continuous on $\mathcal{K}$ by Theorem \ref{translativeMink}, {\rm(v)}.
By Groemer's extension theorem (\cite[Thm. 14.4.2]{Schneider.2008d}) the continuity implies the existence of additive extensions to the convex ring $\mathcal{R}$. By \cite[Thm. 14.4.4]{Schneider.2008d},  the extensions are also measurable. The conditional boundedness follows in both cases from the continuity, since for given $K\in\mathcal{K}$ the set $\{L\in\mathcal{K}:L\subset K\}$ is compact by \cite[Thm. 1.8.4]{Schneider.1993}.
\end{proof}
Now we assume that the particles of $X$ are elements of the convex ring $\mathcal{R}$ and that the grain distribution $\mathbb{Q}$ satisfies the integrability condition
\begin{equation}\label{IntCondBM1}
\int\limits_{\mathcal{R}_0} 2^{N(C)} \lambda(C+\varrho B^n)\mathbb{Q}(dC) <\infty\quad \text{ for some (and hence all) } \varrho>0.
\end{equation}
If the particles are convex, the above integrability condition reduces to \eqref{IntCondBM}.
In \cite[Thm. 9.2.2]{Schneider.2008d}  it is shown that a functional $\varphi: \mathcal{R}\rightarrow \mathbb{R}$ is $\mathbb{Q}$-integrable, if it is translation invariant, additive, measurable and conditionally bounded.
Therefore the $\mathbb{Q}$-integrability  of the Minkowski tensors with the exponent $r=0$ follows directly from Lemma \ref{LemPropForDens} and the translation invariance of $\Phi_j^{0,s}$.
For the mixed Minkowski tensors, Lemma \ref{LemPropForDens}, {\rm(ii)}, yields the $\mathbb{Q}$-integrability with respect to each of the arguments $K_i$, for $2\leq i\leq k$, and only in the case $r=0$ also with respect to the argument $K_1$. This is a consequence of the translation invariance with respect to these arguments. If we want to define densities with respect to several arguments simultaneously, we have to assume that an integrability condition is satisfied.

Let $s\in\mathbb{N}_0$ and $j\in\{0,\ldots,n-1\}$. Then we define the $\Phi_j^{0,s}$-density of $X$ by
\[
\overline{\Phi}_j^{\,0,s}(X):= \gamma \int\limits_{\mathcal{K}_0} \Phi_j^{0,s}(K) \mathbb{Q}(dK).
\]

Let $K_1,\ldots,K_k\in\mathcal{K}$, $r$, $s \in\mathbb{N}_0$, $j\in\{0,\ldots,n-1\}$, $k\in\mathbb{N}$, $m_1,\ldots,m_k\in\{j,\ldots,n\}$ with $m_1+\ldots+m_k=(k-1)n+j$ and $l\in\{2,\ldots,k\}$. If the function
$$
(K_2,\ldots,K_l) \mapsto \ ^j\Phi_{m_1,\ldots,m_k}^{\,r,s}(K_1,\ldots,K_k)
$$
is $\mathbb{Q}^{l-1}$-integrable, we define mixed densities by
\begin{align*}
& \ ^j\overline{\Phi}^{\,r,s}_{m_1,\ldots,m_k}(K_1,X,\ldots,X,K_{l+1},\ldots,\ldots,K_k)\\
&\qquad:= \gamma^{l-1} \int\limits_{\mathcal{K}_0}\ldots\int\limits_{\mathcal{K}_0}  \ ^j\Phi_{m_1,\ldots,m_k}^{\,r,s}(K_1,\ldots,K_k) \mathbb{Q}(dK_2)\ldots \mathbb{Q}(dK_l).
\end{align*}
The mixed density is defined in the same way for other positions of the integration variables, if a corresponding integrability condition is fulfilled. Though, except for the case $r=0$, the first position will be omitted because the integrand is not translation invariant with respect to $K_1$.

Next we  define  densities of the Minkowski tensors for standard random closed sets. Recall from \cite[Def. 9.2.1]{Schneider.2008d}
that a standard random set in $\mathbb{R}^n$ is a random closed set $\tilde{Z}$ in $\mathbb{R}^n$ for which the realizations of $\tilde{Z}$ are a.s.~locally polyconvex, $\tilde{Z}$ is stationary, and
\begin{equation}\label{integrabcond}
\mathbb{E}[2^{N(\tilde{Z}\cap C^n)}]<\infty.
\end{equation}

By Lemma \ref{LemPropForDens} the coordinates of $\Phi_j^{0,s}$ are additive, measurable and conditionally bounded, obviously they are also translation invariant. Hence,
by \cite[Thm. 9.2.1]{Schneider.2008d}  we can  define densities $\overline{\Phi}_j^{\,0,s}(\tilde{Z})$ in the following way.

Let $\tilde{Z}$ be a standard random set, $W\in\mathcal{K}$ with $V_n(W)>0$, $s\in\mathbb{N}_0$,  $j\in\{0,\ldots,n-1\}$. Then the limit
\[
\overline{\Phi}_j^{\,0,s}(\tilde{Z}):= \lim\limits_{\varrho\rightarrow\infty}\frac{\mathbb{E}[\Phi_j^{0,s}(\tilde{Z}\cap \varrho W)]}{V_n(\varrho W)}
\]
exists and satisfies
\[
\overline{\Phi}_j^{\,0,s}(\tilde{Z})=\mathbb{E}[\Phi_j^{0,s}(\tilde{Z}\cap C^n) -\Phi_j^{0,s}(\tilde{Z}\cap \partial^+ C^n) ].
\]
In particular, $\overline{\Phi}_j^{\,0,s}(\tilde{Z})$ is independent of the choice of $W$.



\section{Mean Value and Density Formulas}\label{SectionMeanValDensFormulas}

\subsection{Stationary Boolean Model}
From now on we assume that the grain distribution $\Q$ is concentrated on $\cK_0$. Then the Boolean model $Z$ is a standard random set since the integrability condition \eqref{integrabcond} holds, compare \cite[p. 384]{Schneider.2008d}.
In the described setting the following density formulas for the scalar valued intrinsic volumes have been proven by Weil \cite[Cor. 7.5]{Weil.1990} and are stated in \cite[Thm. 9.1.5]{Schneider.2008d}.
\begin{theorem}[Weil 1990]\label{IntrinsicVolDensFormulas}
Let $Z$ be a stationary Boolean model $Z$ with convex grains. Then
\[\overline{V}_n(Z)=1-\mathrm{e}^{-\overline{V}_n(X)},\]
\[\overline{V}_{n-1}(Z)=\mathrm{e}^{-\overline{V}_n(X)}\overline{V}_{n-1}(X)\]
and
\[
\overline{V}_j(Z)=\mathrm{e}^{-\overline{V}_n(X)}\left[
\overline{V}_j(X)-\sum\limits_{l=2}^{n-j}\frac{(-1)^{l}}{l!}\sum
\limits_{\substack{m_1,\ldots,m_l=j+1\\
m_1+\ldots +m_l=(l-1)n+j}}^{n-1} \ ^j\overline{V}_{m_1,\ldots,m_l}(X,\ldots,X)
\right]
\]
for $j=0,\ldots,n-2$.
\end{theorem}

In this section we first establish connections between mean values of the Min\-kow\-ski tensors of the intersection of $Z$ with a compact, convex window $W$ and the densities of the particle process $X$. For the translation invariant Minkowski tensors, we obtain thus in a second step corresponding relations between the densities of $Z$ and the densities of $X$ which generalize the above theorem for intrinsic volume densities.
\begin{theorem}\label{ExpOfSectWithWindow}
Let $Z$ be a stationary Boolean model in $\mathbb{R}^n$ with convex grains, let $W\in\mathcal{K}$ and $r,s\in\N_0$. Then
\begin{align*}
\mathbb{E}[\Phi_{n-1}^{r,s}(Z\cap W)] = \Phi_{n-1}^{r,s}(W)\left(1-\mathrm{e}^{-\overline{V}_n(X)}\right)+\Phi_n^{r,0}(W)\overline{\Phi}_{n-1}^{\,0,s}(X) \mathrm{e}^{-\overline{V}_n(X)}
\end{align*}
and
\begin{align*}
\mathbb{E}\left[ \Phi_n^{r,0}(Z\cap W)\right]=
\Phi_n^{r,0}(W)\left( 1-\mathrm{e}^{-\overline{V}_n(X)}\right).
\end{align*}
\end{theorem}
\begin{proof}
By \cite[Thm. 9.1.2]{Schneider.2008d}, we have
$
\mathbb{E}[|\Phi_{n-1}^{r,s}(Z\cap W)|]<\infty,
$
where the absolute value and the relation $<$ are applied coordinate-wise and
\begin{equation}\label{n-1}
 \mathbb{E}[\Phi_{n-1}^{r,s}(Z\cap W)]= \sum\limits_{k=1}^{\infty} \frac{(-1)^{k-1}}{k!} \gamma^k
\int\limits_{\mathcal{K}_0}\ldots\int\limits_{\mathcal{K}_0}
\Phi(W,K_1,\ldots,K_k)\mathbb{Q}(dK_1)\ldots\mathbb{Q}(dK_k),
\end{equation}
where
\[ \Phi(W,K_1,\ldots,K_k)
= \int\limits_{(\mathbb{R}^n)^k}\Phi_{n-1}^{r,s}(W\cap(K_1+x_1)\cap\ldots\cap(K_k+x_k))dx_1\ldots dx_k
\]
and the series on the right-hand side of \eqref{n-1} converges absolutely.
In the next step, we apply the iterated translative formula, Corollary \ref{IteratedTranslated}, to get
\begin{multline*}
 \Phi(W,K_1,\ldots,K_k) = \Phi_{n-1}^{r,s}(W)V_n(K_1)\cdots V_n(K_k)
  +\sum\limits_{l=1}^k \Phi_n^{r,0}(W)\\
  \times V_n(K_1)\cdots V_n(K_{l-1})\Phi_{n-1}^{0,s}(K_l)V_n(K_{l+1})\cdots V_n(K_k).
\end{multline*}

Therefore, we obtain
\begin{align*}
\mathbb{E}[\Phi_{n-1}^{r,s}(Z\cap W)]
& = \sum\limits_{k=1}^\infty \frac{(-1)^{k-1}}{k!}
\Big(\Phi_{n-1}^{r,s}(W) \overline{V}_n(X)^k +k \Phi_n^{r,0}(W)\overline{V}_n(X)^{k-1} \overline{\Phi}_{n-1}^{\,0,s}(X)\Big)\\
& = \Phi_{n-1}^{r,s}(W)\big(1-\mathrm{e}^{-\overline{V}_n(X)}\big)+\Phi_n^{r,0}(W)\overline{\Phi}_{n-1}^{\,0,s}(X) \mathrm{e}^{-\overline{V}_n(X)}.
\end{align*}
The existence of the occurring densities follows from the discussion in Section \ref{SectionDefDensities}. The second asserted relation follows in a similar way as the first one.
\end{proof}
%
\begin{cor}\label{Density}
If $s\in\mathbb{N}_0$, then
\[
\overline{\Phi}_{n-1}^{\,0,s}(Z)= \overline{\Phi}_{n-1}^{\,0,s}(X)\mathrm{e}^{-\overline{V}_n(X)}.
\]
\end{cor}
\begin{proof}
 Let $W\in\mathcal{K}$ with $V_n(W)>0$. Then Theorem \ref{ExpOfSectWithWindow} yields
$$
\overline{\Phi}_{n-1}^{\,0,s}(Z)
 = \lim\limits_{\varrho\rightarrow\infty}\frac{\mathbb{E}[\Phi_{n-1}^{0,s}(Z\cap \varrho W)]}{V_n(\varrho W)}
 = \overline{\Phi}_{n-1}^{\,0,s}(X)\mathrm{e}^{-\overline{V}_n(X)}.
$$
\end{proof}

Relations between $\mathbb{E}[\Phi_j^{r,s}(Z \cap W)]$ and densities of $X$ become more complicated  as $j$ is getting smaller.
\begin{theorem}\label{ExpectFormulasGeneral}
Let $Z$ be a stationary Boolean model in $\R^n$ with convex grains, $W\in\cK$ and $r,s\in\N_0$. Then, we obtain
for $n\geq 3$ and $j=1,\ldots,n-2$ that
\begin{align*}
\E\left[ \Phi_j^{r,s}(Z\cap W)\right]
&=\Phi_j^{r,s}(W)\left( 1-\mathrm{e}^{- \overline{V}_n(X)}\right)+ \sum\limits_{m_0=j+1}^{n-1}\;\sum\limits_{l=1}^{m_0-j} \frac{(-1)^{l-1}}{l!} \\
&\quad\quad \times \sum\limits_{\substack{m_1,\ldots,m_l=j\\ m_1+\ldots +m_l=ln+j-m_0}}^{n-1} \ ^j\overline{\Phi}_{m_0,m_1,\ldots,m_l}^{\,r,s}(W,X,\ldots,X) \mathrm{e}^{-\overline{V}_n(X)}\\
& \quad +\Phi_n^{r,0}(W)\sum\limits_{l=1}^{n-j} \frac{(-1)^{l-1}}{l!} \sum\limits_{\substack{m_1,\ldots,m_l=j\\ m_1+\ldots +m_l=(l-1)n+j}}^{n-1} \ ^j\overline{\Phi}_{m_1,\ldots,m_l}^{\,0,s}(X,\ldots,X) \mathrm{e}^{-\overline{V}_n(X)}
\end{align*}
and
for $n\geq 2$ and $j=0$ that
\begin{align*}
\E\left[ \Phi_0^{r,s}(Z\cap W)\right]
&=\Phi_0^{r,s}(W)\left( 1-\mathrm{e}^{- \overline{V}_n(X)}\right)\\
&\quad+ \sum\limits_{m_0=1}^{n-1}\;\sum\limits_{l=1}^{m_0} \frac{(-1)^{l-1}}{l!} \sum\limits_{\substack{m_1,\ldots,m_l=0\\ m_1+\ldots +m_l=ln-m_0}}^{n-1} \ ^0\overline{\Phi}_{m_0,m_1,\ldots,m_l}^{\,r,s}(W,X,\ldots,X) \mathrm{e}^{-\overline{V}_n(X)}\\
& \quad+\mathbf{1}\{s\in 2\N_0\}  \frac{2}{s!\omega_{s+1}}Q^{\frac{s}{2}}  \Phi_n^{r,0}(W)\\
&\quad\quad  \times \sum\limits_{l=1}^{n} \frac{(-1)^{l-1}}{l!} \sum\limits_{\substack{m_1,\ldots,m_l=0\\ m_1+\ldots +m_l=(l-1)n}}^{n-1} \ ^0\overline{V}_{m_1,\ldots,m_l}(X,\ldots,X) \mathrm{e}^{-\overline{V}_n(X)}.
\end{align*}
\end{theorem}
\begin{proof}
By \cite[Thm. 9.1.2]{Schneider.2008d}, we have $\mathbb{E}[|\Phi_j^{r,s}(Z\cap W)|]<\infty$ (see above) and
$$
\mathbb{E}\left[\Phi_j^{r,s}(Z\cap W)\right]
 = \sum\limits_{k=1}^{\infty} \frac{(-1)^{k-1}}{k!} \gamma^k
\int\limits_{\mathcal{K}_0}\ldots\int\limits_{\mathcal{K}_0}
\Phi(W,K_1,\ldots,K_k)\mathbb{Q}(dK_1)\ldots\mathbb{Q}(dK_k), 
$$
where
$$
 \Phi(W,K_1,\ldots,K_k)
= \int\limits_{(\mathbb{R}^n)^k}\Phi_j^{r,s}(W\cap(K_1+x_1)\cap\ldots\cap(K_k+x_k))dx_1\ldots dx_k.
$$

The function
\[\Phi(W,\cdot,\ldots,\cdot): \mathcal{K}^k\rightarrow \mathbb{T}^{r+s}\]
is $\mathbb{Q}^k$-integrable by the proof of \cite[Thm. 9.1.2]{Schneider.2008d}. For $r=s=0$ the mixed Minkowski tensors are nonnegative, real-valued functionals and therefore they are also $\mathbb{Q}^k$-integrable by Theorem \ref{translativeMink}. For general $r$, $s\in\mathbb{N}_0$, we have
\[
|\ ^j\Phi_{m_1,\ldots,m_k}^{r,s}(K_1,\ldots,K_k)|_{\infty} \leq (\max\limits_{x\in K_1} \|x\|)^r \ ^j\Phi_{m_1,\ldots,m_k}^{0,0}(K_1,\ldots,K_k).
\]
This shows the existence of densities for the mixed Minkowski tensors, though only in the case $r=0$ the density can be formed with respect to the first argument.
Now we apply the iterated translative integral formula, Theorem \ref{translativeMink}. Introducing the index $l$ as the number of indices among $m_1,\ldots,m_k$ that are smaller than $n$, we can rearrange the summation to get
\begin{align*}
\E\left[\Phi_j^{r,s}(Z\cap W)\right]
& = \sum\limits_{k=1}^{\infty } \frac{(-1)^{k-1}}{k!}\gamma^{k}
\int\limits_{\cK_0^k} \sum\limits_{\substack{m_0,\ldots,m_k=j
m_0,\ldots,m_k=kn+j}}^n \\&
\qquad\qquad \ ^j\Phi_{m_0,\ldots,m_k}^{r,s}(W,K_1,\ldots,K_k)\Q^k(d(K_1,\ldots,K_k))\\
& = \sum\limits_{k=1}^{\infty } \frac{(-1)^{k-1}}{k!}\gamma^k \sum\limits_{m_0=j}^{n}\sum\limits_{\substack{m_1,\ldots,m_k=j\\ m_1+\ldots + m_k=kn+j-m_0}}^n \\
 &\qquad\qquad \int\limits_{\cK_0^k} \ ^j\Phi_{m_0,\ldots,m_k}^{r,s}(W,K_1,\ldots,K_k)\Q^k(d(K_1,\ldots,K_k))\\
& = \sum\limits_{k=1}^{\infty } \frac{(-1)^{k-1}}{k!}\gamma^k \sum\limits_{m_0=j}^n \sum\limits_{l=\mathbf{1}\{m_0>j\}}^{(m_0-j)\wedge k} \binom{k}{k-l} \sum\limits_{\substack{m_1,\ldots,m_l=j\\ m_1+\ldots m_l=ln+j-m_0}}^{n-1}\\
 & \qquad\qquad\int\limits_{\cK_0^k} \ ^j\Phi^{r,s}_{m_0,m_1,\ldots,m_l,\underbrace{n,\ldots,n}_{k-l \text{ times}}}(W,K_1,\ldots,K_k) \Q^k(d(K_1,\ldots,K_k))\\
& =\sum\limits_{m_0=j}^n \sum\limits_{l=\mathbf{1}\{m_0>j\}}^{m_0-j} \sum\limits_{m=\mathbf{1}\{m_0=j\}}^{\infty } \frac{(-1)^{m+l-1}}{(m+l)!} \binom{m+l}{m} \sum\limits_{\substack{m_1,\ldots,m_l=j\\ m_1+\ldots +m_l=ln+j-m_0}}^{n-1} \\
 &\qquad\qquad
\ ^j\overline{\Phi}_{m_0,m_1,\ldots,m_l}^{\,r,s}(W,X,\ldots,X) \overline{V}_n(X)^m\\
& = \sum\limits_{m_0=j+1}^n \sum\limits_{l=1}^{m_0-j} \frac{(-1)^{l-1}}{l!} \sum\limits_{\substack{m_1,\ldots,m_l=j\\ m_1+\ldots +m_l=ln+j-m_0}}^{n-1} \ ^j\overline{\Phi}_{m_0,\ldots,m_l}^{\,r,s}(W,X,\ldots,X) \mathrm{e}^{-\overline{V}_n(X)}\\
 & \quad +\overline{\Phi}_j^{\,r,s}(W)
\left[ 1-\mathrm{e}^{-\overline{V}_n(X)}
\right].
\end{align*}

A special situation occurs in the case $j=0$, namely the tensor $\Phi_0^{0,s}$, for even $s$, is proportional to a multiple of the metric tensor and for odd $s$ it is equal to the zero tensor.
By applying the representation \cite[(4.2.20)]{Schneider.1993} to the measure $\Lambda_0(K;\mathbb{R}^n\times \cdot)$ and by \cite[(24)]{Schneider.2002}, we have for $K\in\mathcal{K}$ that
\begin{align*}
\Phi_0^{0,s}(K) &= \frac{1}{s!}\frac{\omega_n}{\omega_{n+s}} \int\limits_{\mathbb{R}^n\times S^{n-1}}
u^s \Lambda_0(K;d(x,u))
 = \frac{1}{s!}\frac{1}{\omega_{n+s}}  V_0(K) \int\limits_{S^{n-1}} u^s \mathcal{H}^{n-1}(du)\\
& = \frac{2}{s! \omega_{s+1}}V_0(K) Q^{\frac{s}{2}} \mathbf{1}\{s\in 2\mathbb{N}_0\}.
\end{align*}
Furthermore, if we apply the translative integral formula Theorem \ref{translativeMink} for the tensor $\Phi_0^{0,0}=V_0$ and the tensor $\Phi_0^{0,s}$, we obtain by a comparison of the right-hand sides and by using  the fact that  the mixed tensors have different degrees of  homogeneity in the arguments $K_1,\ldots,K_k$ that
\[
\ ^0\Phi_{m_1,\ldots,m_k}^{0,s}(K_1,\ldots,K_k)= \mathbf{1}\{s\in 2\mathbb{N}_0\}\,\frac{2}{s! \omega_{s+1}} Q^{\frac{s}{2}}\,\ ^0V_{m_1,\ldots,m_k}(K_1,\ldots,K_k).
\]
\end{proof}
In the subsequent corollary we state the mean value formulas in the special case of three dimensions. The result in two dimensions has already been formulated in Corollary \ref{ExpOfSectWithWindowDimTwo}.
\begin{cor}\label{ExpOfSectWithWindowDimThree}
Let $Z$ be a stationary Boolean model in $\mathbb{R}^3$ with convex grains, let $W\in\mathcal{K}$ and $r,s\in\N_0$.
Then
\begin{align*}
 \mathbb{E}\left[\Phi_0^{r,s}(Z\cap W)\right]
& = \Phi_0^{r,s}(W)\left(1-\mathrm{e}^{-\overline{V}_3(X)}\right)+\mathrm{e}^{-\overline{V}_3(X)}
\Big[- \frac{1}{2}  \ ^0\overline{\Phi}^{\,r,s}_{2,2,2}(W,X,X)\\
& \quad + \ ^0\overline{\Phi}^{\,r,s}_{1,2}(W,X) + \ ^0\overline{\Phi}^{\,r,s}_{2,1}(W,X)\\
&\quad + \mathbf{1}\{s \in 2\mathbb{N}_0\}\, \frac{2}{s! \omega_{s+1}}\, Q^{\frac{s}{2}}   \Big(  \Phi_3^{\,r,0}(W)\overline{V}_0(X)\\
& \quad - \Phi_3^{r,0}(W)\ \ ^0\overline{V}_{1,2}(X,X) + \frac{1}{6} \Phi_3^{r,0}(W)  \ ^0\overline{V}_{2,2,2}(X,X,X) \Big) \Big];\\
 \mathbb{E}\left[\Phi_1^{r,s}(Z\cap W)\right]
& = \Phi_1^{r,s}(W)\left(1-\mathrm{e}^{-\overline{V}_3(X)}\right)+\mathrm{e}^{-\overline{V}_3(X)}\Big[\Phi_3^{r,0}(W)\overline{\Phi}_1^{\,0,s}(X)\\
& \quad + \ ^1\overline{\Phi}^{\,r,s}_{2,2}(W,X)- \frac{1}{2} \Phi_3^{r,0}(W) \ ^1\overline{\Phi}^{\,0,s}_{2,2}(X,X)\Big];\\
\mathbb{E}\left[\Phi_{2}^{r,s}(Z\cap W)\right]& = \Phi_{2}^{r,s}(W)\left(1-\mathrm{e}^{-\overline{V}_3(X)}\right)+\Phi_3^{r,0}(W) \overline{\Phi}_{2}^{\,0,s}(X) \mathrm{e}^{-\overline{V}_3(X)};\\
\mathbb{E}[ \Phi_3^{r,0}(Z\cap W)]&=
\Phi_3^{r,0}(W)\left( 1-\mathrm{e}^{-\overline{V}_3(X)}\right).
\end{align*}
\end{cor}
We obtain the following density formulas in general dimension.

\begin{cor}\label{DensityArbDim}
Let $Z$ be  a stationary Boolean model  in $\mathbb{R}^n$ with convex grains and $s\in\N_0$. Then we obtain
\[\overline{\Phi}_{n-1}^{\,0,s}(Z)=\mathrm{e}^{-\overline{V}_n(X)}\overline{\Phi}_{n-1}^{\,0,s}(X)\]
and for $n\geq 3$ and $j=1,\ldots,n-2$ that
$$
\overline{\Phi}_j^{\,0,s}(Z)
= \mathrm{e}^{-\overline{V}_n(X)}\left(\overline{\Phi}_j^{\,0,s}(X)-\sum\limits_{l=2}^{n-j}
\frac{(-1)^l}{l!} \sum\limits_{\substack{m_1,\ldots,m_l=j+1\\ m_1+\ldots +m_l=(l-1)n+j}}^{n-1}\hspace{0.001cm} ^j\overline{\Phi}^{\,0,s}_{m_1,\ldots,m_l}(X,\ldots,X)
\right)
$$
and for $n\geq 2$ and $j=0$ that
\begin{align*}
\overline{\Phi}_0^{0,s}(Z)
&=\mathbf{1}\{s\in 2\N_0\}\, \mathrm{e}^{-\overline{V}_n(X)} \frac{2}{s!\omega_{s+1}}Q^{\frac{s}{2}}\\
&\quad\times \left(\overline{V}_0(X) -\sum\limits_{l=2}^{n} \frac{(-1)^{l}}{l!} \sum\limits_{\substack{m_1,\ldots,m_l=1\\ m_1+\ldots +m_l=(l-1)n}}^{n-1} \ ^0\overline{V}_{m_1,\ldots,m_l}(X,\ldots,X)
\right)\\
&=
\mathbf{1}\{s\in 2\N_0\}  \frac{2}{s!\omega_{s+1}}Q^{\frac{s}{2}} \overline{V}_0(Z).
\end{align*}
\end{cor}
\begin{proof}
Let $W\in \cK$ with $V_n(W)>0$. We consider for $\varrho>0$ in the formulas from Theorem \ref{ExpectFormulasGeneral} the dilated window $\varrho W$ instead of $W$ and divide by its volume $V_n(\varrho W)$. Due to the homogeneity properties of the mixed Minkowski tensors, the summands with $m_0<n$ vanish asymptotically as $\varrho$ goes to infinity. For the second relation we use Theorem \ref{IntrinsicVolDensFormulas}.
\end{proof}
The subsequent Corollary \ref{DensityDimTwoThree} states the density formulas in two and three dimensions. The cases $n=2$ and $j=1$ respectively $n=3$ and $j=2$ are already contained in Corollary \ref{Density} and therefore not displayed again.
\begin{cor}\label{DensityDimTwoThree}
Let $s\in\mathbb{N}_0$.
\begin{itemize}
\item [Then, for $n=2$ and $j=0$, we have]
\begin{align*}
 \overline{\Phi}_0^{\,0,s}(Z) &= \mathbf{1}\{s \in 2\mathbb{N}_0\}\,\frac{2}{s! \omega_{s+1}}\, Q^{\frac{s}{2}} \,\overline{V}_0(Z)\\
&=  \mathbf{1}\{s \in 2\mathbb{N}_0\}\,\frac{2}{s! \omega_{s+1}}\, Q^{\frac{s}{2}}
\left(\overline{V}_0(X)-\frac{1}{2} \ ^0\overline{V}_{1,1}(X,X)\right)\mathrm{e}^{-\overline{V}_2(X)} ;
\end{align*}
\item [for $n=3$ and $j=0$, we have]
\begin{align*}
\overline{\Phi}_0^{\,0,s}(Z)& =  \mathbf{1}\{s \in 2\mathbb{N}_0\}\,\frac{2}{s! \omega_{s+1}}\, Q^{\frac{s}{2}}\, \overline{V}_0(Z)\\
& = \mathbf{1}\{s \in 2\mathbb{N}_0\}\,\frac{2}{s! \omega_{s+1}}\, Q^{\frac{s}{2}}
\Big(\overline{V}_0(X)
- \ ^0\overline{V}_{1,2}(X,X)+\frac{1}{6} \ ^0\overline{V}_{2,2,2}(X,X,X)\Big) \mathrm{e}^{-\overline{V}_3(X)}  ;
\end{align*}
\item [for $n=3$ and $j=1$, we have]
\begin{align*}
& \overline{\Phi}_1^{\,0,s}(Z)=\Big(\overline{\Phi}_1^{\,0,s}(X) -\frac{1}{2} \ ^1\overline{\Phi}^{\,0,s}_{2,2}(X,X)\Big)\mathrm{e}^{-\overline{V}_3(X)}.
\end{align*}
\end{itemize}
\end{cor}
\begin{remark}\label{EulerCaseNotThatInteresting}
 A comparison of the previous Corollary \ref{DensityArbDim} with Theorem \ref{IntrinsicVolDensFormulas} shows that in the case $j=0$ the Minkowski tensor densities do not contain more information than the scalar valued densities $\overline{V}_0(Z)$. Though we would like to point out that this is indeed not the case for the corresponding mean value formulas for finite section window $W$, compare Corollary \ref{ExpOfSectWithWindowDimTwo} respectively Corollary \ref{ExpOfSectWithWindowDimThree} in the case $j=0$. Namely, if for pairwise distinct  $\varrho_0,\ldots,\varrho_{n}>0$ the mean values
$\mathbb{E}\left[\Phi_0^{r,s}(Z\cap \varrho_k W)\right]$, for $ k=0,\ldots, n$,
are known, we can separate the summands of different homogeneity degree in the right-hand side of the corresponding equations by merely solving a system of linear equations. In particular, if additionally the density $\overline{V}_n(X)$ is known, we obtain the density $ \ ^0\overline{\Phi}^{\,r,s}_{1,1}(W,X)$ in the case $n=2$ and the density
$\ ^0\overline{\Phi}^{\,r,s}_{2,2,2}(W,X,X)$ in the case $n=3$.
\end{remark}
%

\begin{remark}\label{remPolyconvexGrains}
In the statement of the density formulas we restricted to stationary Boolean models with convex grains. In \cite{Weil.2001} density formulas for the mixed functionals of translative integral geometry are established for Boolean models with polyconvex grains satisfying the integrability condition \eqref{IntCondBM1}. Condition \eqref{IntCondBM1} is also sufficient for extending the proof of Theorem \ref{ExpectFormulasGeneral} to the setting of stationary Boolean models with polyconvex grains. We shall not go into details here. But observe that it can be shown that a stationary Boolean model with grain distribution $\mathbb{Q}$ concentrated on the convex ring and intensity $\gamma >0$, which satisfies the integrability condition \eqref{IntCondBM1}
is a standard random set. Furthermore, all necessary integrability properties of involved mixed Minkowski tensors with respect to multiple product measures of $\mathbb{Q}$ can be handled.
\end{remark}

\subsection{Densities of Isotropic Standard Random Sets}
In this section we consider an isotropic standard random set $\tilde{Z}$. We shall see that in this case the densities of the Minkowski tensors are just constant multiples of the densities of the intrinsic volumes.

\begin{prop}\label{isotropicStandardRandSet}
Let $\tilde{Z}$ be an isotropic standard random set, $j\in \{0,\ldots,n-1\}$ and $s\in\mathbb{N}_0$. Then
\[ \overline{\Phi}_j^{\,0,s}(\tilde{Z})=\mathbf{1}\{s\in 2\N_0\}\tilde{\alpha}_{n,j,s} Q^{\frac{s}{2}} \overline{V}_j(\tilde{Z}),
\]
where
\[\tilde{\alpha}_{n,j,s}:=\frac{1}{(4\pi)^{\frac{s}{2}}(\frac{s}{2})!} \frac{\Gamma(\frac{n-j+s}{2})\Gamma(\frac{n}{2})}{\Gamma(\frac{n+s}{2})\Gamma(\frac{n-j}{2})}.\]
\end{prop}
\begin{proof}
First, we have
\begin{align*}
 \mathbb{E}\left[\Phi_j^{0,s}(\tilde{Z}\cap B^n)\right]
& = \int\limits_{SO_n}\mathbb{E}\left[\Phi_j^{0,s}(\vartheta \tilde{Z}\cap B^n)\right]\nu(d\vartheta)\\
& = \mathbb{E}\left[\;\int\limits_{SO_n}\frac{1}{s!}\frac{\omega_{n-j}}{\omega_{n-j+s}} \int\limits_{\Sigma}u^s\Lambda_j\left(\vartheta \tilde{Z}\cap B^n;d(x,u)\right)\nu(d\vartheta)\right].
\end{align*}
Now we define a measure $\mu$ on $S^{n-1}$ by
\[
\mu(A):= \int\limits_{SO_n}\Lambda_j\left(\vartheta \tilde{Z} \cap B^n; \mathbb{R}^n\times A\right)\nu(d\vartheta)
\]
for $A\in\mathcal{B}(S^{n-1})$.
The measure $\mu$ is $SO_n$-invariant because of the rotation covariance of $\Lambda_j$ and the invariance properties of the Haar measure on $S^{n-1}$. For $\vartheta \in SO_n$ and $A\in\mathcal{B}(S^{n-1})$ we obtain
$\mu(\vartheta A) = \mu(A)$.
Hence, $\mu$ is a multiple of the Haar measure on $S^{n-1}$ and $\mu\left(S^{n-1}\right)= V_j\left(\tilde{Z}\cap B^n\right)$.
We deduce
\[
\mu = \frac{V_j\left(\tilde{Z}\cap B^n\right)}{n\kappa_n}\,\mathcal{H}^{n-1}\llcorner S^{n-1}.
\]
Now it follows that
\begin{align*}
 \mathbb{E}\left[\Phi_j^{0,s}\left(\tilde{Z}\cap B^n\right)\right]
& =  \frac{1}{s!}\mathbb{E}\left[\frac{\omega_{n-j}}{\omega_{n-j+s}}
\int\limits_{S^{n-1}}u^s\mu(du)\right]\\
& = \mathbb{E}\left[V_j\left( \tilde{Z}\cap B^n\right)\right]\frac{1}{s!}\frac{\omega_{n-j}}{\omega_{n-j+s}}\frac{1}{n\kappa_n}\int\limits_{S^{n-1}}u^s
\mathcal{H}^{n-1}(du)\\
& =\mathbf{1}\{s\in 2\N_0\} \mathbb{E}\left[V_j\left( \tilde{Z}\cap B^n\right)\right]\frac{1}{s!}\frac{\omega_{n-j}}{\omega_{n-j+s}}\frac{1}{n\kappa_n} 2\frac{\omega_{s+n}}{\omega_{s+1}} Q^{\frac{s}{2}},
\end{align*}
where the last equality follows by \cite[(24)]{Schneider.2002}.
Therefore, we have
\begin{align*}
 \overline{\Phi}_j^{\,0,s}(\tilde{Z})
& = \lim\limits_{r\rightarrow\infty}\frac{\mathbb{E}[\Phi_j^{0,s}(\tilde{Z}\cap rB^n)}{r^n\kappa_n}\\
& =\mathbf{1}\{s\in 2\N_0\} \overline{V}_j(\tilde{Z}) \frac{1}{n\kappa_n}\frac{1}{s!}\frac{\omega_{n-j}}{\omega_{n-j+s}} 2 \frac{\omega_{s+n}}{\omega_{s+1}} Q^{\frac{s}{2}} \\
& = \mathbf{1}\{s\in 2\N_0\} \frac{\Gamma(\frac{n}{2})\Gamma(\frac{n-j+s}{2})\Gamma(\frac{s+1}{2})}{\pi^{\frac{s+1}{2}}\Gamma(s+1) \Gamma(\frac{n-j}{2})\Gamma(\frac{n+s}{2})} Q^{\frac{s}{2}} \overline{V}_j(\tilde{Z})\\
& =\mathbf{1}\{s\in 2\N_0\} \tilde{\alpha}_{n,j,s} Q^{\frac{s}{2}} \overline{V}_j(\tilde{Z}),
\end{align*}
where the last line follows by Legendre's relation; see \cite{Artin1}.
\end{proof}
\begin{remark}
If the grain distribution $\Q$ is rotation invariant the Boolean model $Z$ is isotropic.
Then Proposition \ref{isotropicStandardRandSet} implies, for $j\in \{0,\ldots,n-1\}$ and $s\in\mathbb{N}_0$, that
\[
\overline{\Phi}_j^{\,0,s}(Z)
=\mathbf{1}\{s\in 2\N_0\}Q^{\frac{s}{2}}\tilde{\alpha}_{n,j,s}\overline{V}_j(Z).
\]
This connection may offer a possibility to test the isotropy of a Boolean model.
\end{remark}

\section{Examples}\label{SectionExamples}
\subsection{Planar Parametric Non-isotropic Boolean Model}
In this subsection we apply the formulas from Corollary \ref{Density} and Corollary \ref{DensityDimTwoThree} to a parametric class of planar Boolean models studied in \cite[Sect. 2.2]{SchroderTurk.2011} with ellipse particles.
We shall see that for this easy parametric model the obtained results allow to extract useful information from observations of the Boolean model.

 For $\alpha\in[0,\infty], \gamma>0$ and $E\in \cK_0$, let $Z_{\alpha,\gamma,E}$ be a stationary Boolean model with intensity $\gamma$ and the grains obtained by rotating $E$ by a random angle $\theta\in[0,2\pi)$.
For $\alpha<\infty$ the random angle $\theta$ has the probability density
\[
f_{\alpha}(\theta)=c(\alpha)\; |\cos\theta|^{\alpha}, \quad \text{ for } \theta\in[0,2\pi),
\]
with
\[
c(\alpha):=\frac{\Gamma(1+\frac{\alpha}{2})}{2\sqrt{\pi} \Gamma(\frac{\alpha+1}{2})},
\]
that is, the grain distribution of $Z_{\alpha,\gamma,E}$ is
\[
\mathbb{Q}(\cdot)=  \int\limits_{0}^{2\pi} \mathbf{1}\{\vartheta(\theta) E \in \cdot\}  f_\alpha(\theta)\,d\theta,
\]
where $\vartheta(\theta)\in SO(2)$ is the rotation by the angle $\theta$. The grain distribution of $Z_{\infty,\gamma,E}$ is $\mathbb{Q}=\delta_E$.
In the following, we call $E$ the base grain and $\alpha$ the orientation parameter of the Boolean model.
We specify in this particular case the formulas for the densities obtained in Corollary \ref{Density} and Corollary \ref{DensityDimTwoThree}. For this we have to determine $\overline{\Phi}_1^{\,0,s}(X), s\in\N_0, \overline{V}_0(X)$ and $\ ^0\overline{V}_{1,1}(X,X)$.
Starting with the density of the surface tensor we obtain for $s\in\N_0$ that
\begin{align*}
\overline{\Phi}_1^{\,0,s}(X) &= \gamma \int\limits_{\mathcal{K}_0}\Phi_1^{\,0,s}(K)\mathbb{Q}(dK) = \gamma \int\limits_{0}^{2\pi}\Phi_1^{\,0,s}(\vartheta(\theta) E)f_\alpha(\theta)d\theta.
\end{align*}
In the following we identify a $p$-tensor with an element of $\mathbb{R}^{n^p}$  in the usual way.
We obtain by \cite[(8)]{Mickel.2012}, for $s\in\mathbb{N}_0$ and $i_1,\ldots,i_s\in\{1,2\}$, that
\begin{align*}
\left(\Phi_1^{0,s}(\vartheta(\theta) E)\right)_{i_1,\ldots,i_s}& =\sum\limits_{j_1,\ldots,j_s=1}^2 (\vartheta(\theta))_{i_1,j_1}\cdots(\vartheta(\theta))_{i_s,j_s} \left(\Phi_1^{0,s}(E)\right)_{j_1,\ldots,j_s}
\end{align*}
and therefore, for $0\leq l\leq s$, that
\begin{align}
&\left(\overline{\Phi}_1^{\;0,s}(X)\right)_{\underbrace{1,\ldots,1}_
{l \text{ times}},\underbrace{2,\ldots,2}_{s-l \text{ times}}}\nn\\
&= \gamma \sum\limits_{j_1,\ldots,j_s=1}^2\,\int\limits_{0}^{2\pi}
(\vartheta(\theta))_{1,j_1}\cdots(\vartheta(\theta))_{1,j_l} \vartheta(\theta))_{2,j_{l+1}}\cdots(\vartheta(\theta))_{2,j_s} \,f_\alpha\,(\theta) d\theta\nn\\
&\quad \times \left(\Phi_1^{0,s}(E)\right)_{j_1,\ldots,j_s}\nn\\
& =
\mathbf{1}\{s \text{ even}\}\,\gamma \sum\limits_{\substack{j=0\\j+l \text{ even}}}^s \sum\limits_{k=0\vee (j-s+l)}^{j\wedge l}(-1)^{l-k}\binom{l}{k}\binom{s-l}{j-k} \prod\limits_{m=1}^{\frac{l+j-2k}{2}}(2m-1)\nn\\
 & \quad \times \prod\limits_{m=1}^{\frac{s-l-j+2k}{2}} (\alpha+2m-1) \prod\limits_{m=1}^{s/2}(\alpha+2m)^{-1}  \left( \Phi_1^{0,s}(E) \right)_{\underbrace{1,\ldots,1}_{j \text{ times}},\underbrace{2,\ldots,2}_{s-j \text{ times}}}, \label{MinkTensDensNonIso}
\end{align}
since we obtain for the integral prefactor in the second line of the above equation for $s_1,\ldots,s_4, s\in\mathbb{N}_0$ with $s_1+\ldots +s_4=s$ that
\begin{align}
&\int\limits_{0}^{2\pi} (\vartheta(\theta))^{s_1}_{1,1}(\vartheta(\theta))^{s_2}_{1,2} (\vartheta(\theta))^{s_3}_{2,1} (\vartheta(\theta))^{s_4}_{2,2} f_\alpha\,(\theta) d\theta\nn\\
& =
c(\alpha) \int\limits_{0}^{2\pi}  (\cos\theta)^{s_1+s_4} |\cos \theta|^{\alpha}(-\sin\theta)^{s_2}(\sin\theta)^{s_3}  d\theta\nn\\
&
=\begin{cases}\begin{array}{lll}
0,& \text{if } s_1+s_4 \text{ or } s_2+s_3 \text{ is odd},\\
(-1)^{s_2}\frac{\prod\limits_{m=1}^{(s_2+s_3)/2}\left(2m-1\right) \prod\limits_{m=1}^{(s_1+s_4)/2}\left(\alpha+2m-1\right)}{ \prod\limits_{m=1}^{s/2} (\alpha+2m)}
,& \text{otherwise},\nn\\
\end{array}\end{cases}
\end{align}
by the symmetry properties of sine and cosine and since
\begin{equation*}\label{eqGamma}
\int\limits_0^{\frac{\pi}{2}}(\sin\varphi)^a (\cos\varphi)^b d\varphi = \frac{1}{2}\frac{\Gamma(\frac{a+1}{2})\Gamma(\frac{b+1}{2})} {\Gamma(\frac{a+b+2}{2})},
\end{equation*}
for $a,b>-1$; see \cite[(5.6)]{Artin1} or \cite[(12.42)]{Whittaker.1996}. In the case $s=2$ equation
 \eqref{MinkTensDensNonIso} simplifies to

\begin{equation}\label{DensNonIsoBMs=2}
\overline{\Phi}_1^{\,0,2}(X)=
\frac{\gamma}{\alpha+2}\scalebox{0.7}{$\begin{pmatrix} (\alpha+1)\left(\Phi_1^{0,2}(E)\right)_{1,1}+ \left(\Phi_1^{0,2}(E)\right)_{2,2} & \;
\alpha \left(\Phi_1^{0,2}(E)\right)_{1,2}\\
&  \\
\alpha \left(\Phi_1^{0,2}(E)\right)_{1,2} & \;  \left(\Phi_1^{0,2}(E)\right)_{1,1}+(\alpha+1) \left(\Phi_1^{0,2}(E)\right)_{2,2}\\
\end{pmatrix}.$}
\end{equation}

On the other hand, we obtain for the mixed density
\begin{align}
 ^0\overline{V}_{1,1}(X,X)
 &=\gamma^2 \int\limits_{0}^{2\pi}\int\limits_{0}^{2\pi}
\ ^0V_{1,1}(\vartheta(\theta_1) E, \vartheta(\theta_2)E)f_\alpha(\theta_1) f_\alpha(\theta_2)d\theta_1 d\theta_2\nn\\
& = \gamma^2 \int\limits_{0}^{2\pi}\int\limits_{0}^{2\pi}
\ ^0V_{1,1}(\vartheta(\theta_1-\theta_2)E, E)f_\alpha(\theta_1) f_\alpha(\theta_2)d\theta_1 d\theta_2\nn\\
&  = \gamma^2 c(\alpha)^2 \int\limits_{0}^{2\pi}\int\limits_{0}^{2\pi}
\ ^0V_{1,1}(\vartheta(\theta_1)E, E)\, |\cos(\theta_1+\theta_2)|^{\alpha} |\cos(\theta_2)|^{\alpha}\, d\theta_1 d\theta_2,\label{mixedDensityGeneral}
\end{align}
where we have used that $\ ^0V_{1,1}$ is invariant with respect to simultaneous rotations of its arguments and that the integrand is $2\pi$-periodic with respect to $\theta_1$.

Furthermore it follows from \cite[Cor. 9.2]{Weil.2001b}
and the rotation covariance of the support measures  for $\theta\in[0,2\pi]$ that
\begin{align}
\ ^0V_{1,1}(\vartheta(\theta)E, E)\nn
& = \frac{2}{\pi} \int\limits_{\mathbb{R}^2\times S^1}\int\limits_{\mathbb{R}^2\times S^1} \alpha\left(\vartheta(\theta)u_1,u_2\right) \sin\left(\alpha\left(\vartheta(\theta)u_1,u_2\right)\right)\nn\\
&\quad\quad\quad\quad\quad\quad\quad\quad\Lambda_1(E;d(x_1,u_1))
\Lambda_1(E;d(x_2,u_2)),\label{mixedVolumeGeneral}
\end{align}
where $\alpha(u_1,u_2)\in[0,\pi]$ denotes the smaller angle between $u_1,u_2\in S^1.$

\begin{remark}
Assume that the above parametric Boolean model is observed and the densities
\[
\overline{\Phi}_1^{\,0,2}(Z_{\alpha,\gamma,E}) \text{ and } \overline{V}_2(Z_{\alpha,\gamma,E})
\]
are therefore known. Is it possible to obtain the parameters $\alpha$ and $\gamma$ from the above densities of the Boolean model?
To see that this is indeed the case, we use Corollary \ref{Density} to obtain
\begin{equation}\label{ParBMSurfaceTensRel}
\overline{\Phi}_1^{\,0,2}(Z_{\alpha,\gamma,E}) = \overline{\Phi}_1^{\,0,2}(X)\, \mathrm{e}^{-\overline{V}_2(X)}
\end{equation}
and (by Theorem \ref{IntrinsicVolDensFormulas})
\begin{equation}\label{volumeFrac}
\overline{V}_2(Z_{\alpha,\gamma,E}) = 1 - \mathrm{e}^{-\overline{V}_2(X)} = 1- \mathrm{e}^{-\gamma V_2(E)}.
\end{equation}
Thus, \eqref{volumeFrac} yields
\begin{equation}\label{FormulaGamma}
\gamma = -\frac{\ln\left(1- \overline{V}_2(Z_{\alpha,\gamma,E})\right)}{V_2(E)}
\end{equation}
and, by \eqref{DensNonIsoBMs=2} and \eqref{ParBMSurfaceTensRel},
\begin{equation}\label{FormulaAlpha}
\alpha = \frac{\gamma\left( \left(\Phi_{1}^{0,2}(E)\right)_{1,1}+\left(\Phi_{1}^{0,2}(E)\right)_{2,2}\right)-2\,\mathrm{e}^{\gamma V_2(E)} \left(\overline{\Phi}_1^{\,0,2}(Z_{\alpha,\gamma,E})\right)_{1,1}}{\mathrm{e}^{\gamma V_2(E)}\left(\overline{\Phi}_1^{\,0,2}(Z_{\alpha,\gamma,E})\right)_{1,1}- \gamma \left(\Phi_{1}^{0,2}(E)\right)_{1,1}}.
\end{equation}
In Subsection \ref{EstOfModelParameters} we use equation \eqref{FormulaGamma} and \eqref{FormulaAlpha} to define estimators for the intensity $\gamma$ and the orientation parameter $\alpha$ and test their performance in a simulation study.
\end{remark}
\begin{remark}
In \cite[Sect. 2.2]{SchroderTurk.2011} the Boolean model $Z_{\alpha,\gamma,E}$ with the base grain $E$ being an ellipse is considered. Pixelized realizations of $[0,1]^2\cap Z_{\alpha,\gamma,E}$ are used as input for testing the performance of real-valued characteristics derived from Minkowski tensors. More precisely, a so-called anisotropy index $\beta_1^{\ast\,0,2}$ is introduced, which is defined by
\[
\beta_1^{\ast\,0,2}:=
\frac{\left(\Phi_1^{0,2}\left(Z_{\alpha,\gamma,E}; [0,1]^2\right)\right)_{1,1}}{\left(\Phi_1^{0,2}\left(Z_{\alpha,\gamma,E};[0,1]^2\right)\right)_{2,2}},
\]
where
\[
\Phi_1^{0,2}\left(Z_{\alpha,\gamma,E};[0,1]^2\right)= \frac{1}{2}\frac{\omega_{1}}{\omega_{3}} \int\limits_{\mathbb{R}^2\times S^1} \mathbf{1}_{[0,1]^2}(x)\;u^2 \Lambda_1(Z_{\alpha,\gamma,E};d(x,u)).
\]
In \cite[2.2]{SchroderTurk.2011}, $\langle\beta_1^{\ast\, 0,2} \rangle$
denotes the mean value obtained by averaging $\beta_1^{\ast\, 0,2}$ over several realizations of $Z_{\alpha,\gamma,E}$ and it is observed that for $\alpha=0$, that is, in the isotropic case,
we have $\langle\beta_1^{\ast\, 0,2} \rangle=1$.

Furthermore, $\langle\beta_1^{\ast\, 0,2} \rangle$
seems to be constant as function of the volume fraction\\ $\overline{V}_2(Z_{\alpha,\gamma,E})$.
Unfortunately, we are right now not able to explain these observations. But if instead of taking the mean value of $\beta_1^{\ast\, 0,2}$, the mean value is taken separately for the denominator and nominator, that is, if
\begin{equation}\label{fracMeanvalues}
\frac{\left\langle \left(\Phi_1^{0,2}\left(Z_{\alpha,\gamma,E}; [0,1]^2\right)\right)_{1,1} \right\rangle}{\left\langle\left(\Phi_1^{0,2}\left(Z_{\alpha,\gamma,E};[0,1]^2\right)\right)_{2,2}\right\rangle}
\end{equation}
is considered, our previous results can be used to obtain some insight.
The quantity \eqref{fracMeanvalues} can be considered as an estimator of
\begin{equation}\label{fracExpectedValues}
\frac{\mathbb{E}\left[\left(\Phi_1^{0,2}\left(Z_{\alpha,\gamma,E}; [0,1]^2\right)\right)_{1,1}\right]}{\mathbb{E}\left[\left(\Phi_1^{0,2}\left(Z_{\alpha,\gamma,E};[0,1]^2\right)\right)_{2,2}\right]}.
\end{equation}
Using the fact that the curvature measures are locally determined, Theorem \ref{TransIntFormSupMeas}, {\rm(ix)}, can be shown to hold also for the additive extensions, and by proceeding as in the proof of Theorem \ref{ExpOfSectWithWindow} and by the special case $j=1, n=2$ of Theorem \ref{TransIntFormSupMeas}, we obtain that
\begin{align*}
\mathbb{E}\left[\Phi_1^{0,2}\left(Z_{\alpha,\gamma,E};[0,1]^2\right)\right]
&= \mathbb{E}\left[\Phi_1^{0,2}\left(Z_{\alpha,\gamma,E}\cap 2B^2;[0,1]^2\right)\right]
 = \overline{\Phi}_1^{\,0,2}(X)\, \mathrm{e}^{-\overline{V}_2(X)}\\
&= \overline{\Phi}_1^{\,0,2}(X)\, \mathrm{e}^{-\gamma V_2(E)}.
\end{align*}
Therefore, by \eqref{DensNonIsoBMs=2}, we get
\begin{align*}
\frac{\mathbb{E}\left[\left(\Phi_1^{0,2}\left(Z_{\alpha,\gamma,E}; [0,1]^2\right)\right)_{1,1}\right]}{\mathbb{E}\left[\left(\Phi_1^{0,2}\left(Z_{\alpha,\gamma,E};[0,1]^2\right)\right)_{2,2}\right]}
 = \frac{\left(\overline{\Phi}_1^{\,0,2}(X)\right)_{1,1}}{\left(\overline{\Phi}_1^{\,0,2}(X)\right)_{2,2}}
= \frac{(\alpha+1)\left( \Phi_1^{0,2}(E)\right)_{1,1}+\left( \Phi_{1}^{0,2}(E)\right)_{2,2}}{\left( \Phi_1^{0,2}(E)\right)_{1,1}+(\alpha+1)\left( \Phi_1^{0,2}(E)\right)_{2,2}}.
\end{align*}
Hence, in the isotropic case ($\alpha =0$) the ratio in \eqref{fracExpectedValues} is  equal to $1$. Moreover, the quantity \eqref{fracExpectedValues} is always independent of the volume fraction $\overline{V}_2(Z_{\alpha,\gamma,E})$, since the volume fraction depends by \eqref{volumeFrac} only on the intensity $\gamma$ and not on the parameter $\alpha$. It is interesting and should be investigated further why these properties are also observed for the quantity $ \langle \beta_1^{\ast\,0,2}\rangle$ in \cite[Sect. 2.2]{SchroderTurk.2011}.
\end{remark}

\begin{remark}\label{RemarkSmoothBaseGrain}
For a smooth base grain $E\in C^2_+$ we obtain special formulas since the support measure $\Lambda_1$ can be represented as an integral over the unit sphere weighted with the curvature radius of $E$ (see \eqref{ReprLambda1AsIntOverUnitSphere}) or as an integral over the boundary of $E$ (see \eqref{ReprLambda1AsIntOverBoundary}). We use the abbreviation
\begin{equation}\label{NotationUofPhi}
u(\alpha):=\begin{pmatrix}\cos(\alpha)\\ \sin(\alpha)\end{pmatrix}, \quad \alpha\in \R,
\end{equation}
and the notation $r(E,u)$ for the radius of curvature of $E$ at a point $x\in \partial E$ with outer normal $u\in S^1$. The representations \cite[(4.2.19) and (4.2.20)]{Schneider.1993} of the curvature respectively area measure for smooth convex bodies lead to
\begin{equation}\label{ReprLambda1AsIntOverUnitSphere}
\Lambda_1(E;\cdot)= \frac{1}{2}\int\limits_{S^{1}}\mathbf{1}\{(x(u),u)\in \cdot\}r(E,u)\mathcal{H}^1(du),
\end{equation}
respectively
\begin{equation}\label{ReprLambda1AsIntOverBoundary}
\Lambda_1(E;\cdot)= \frac{1}{2}\int\limits_{\partial E}\mathbf{1}\{(x,u(x))\in \cdot\}\mathcal{H}^1(dx),
\end{equation}
where for $u\in S^1$ we denote by $x(u)$ the unique boundary point in $\partial E$ with outer normal $u$ and, for $x\in \partial E$, we denote by $u(x)$ the outer normal of $E$ at $x$.
If a parametrization of $\partial E$ is known, \eqref{ReprLambda1AsIntOverBoundary} can be used to determine $\Phi_1^{0,s}(E)$ for $s\in\N_0$, and via equation \eqref{MinkTensDensNonIso} then also $\overline{\Phi}_1^{0,s}(X)$. On the other hand, \eqref{ReprLambda1AsIntOverUnitSphere} can be used to determine $\ ^0 V_{1,1}(\vartheta(\theta)E,E)$, and then via \eqref{mixedDensityGeneral} also $\ ^0\overline{V}_{1,1}(X,X)$; see \eqref{MixedDensSmoothGrain}. In fact, observe that it follows from \eqref{mixedVolumeGeneral} that
\begin{align*}
 ^0V_{1,1}(\vartheta(\theta)E, E)
& = \frac{1}{2\pi} \int\limits_0^{2\pi}\int\limits_0^{2\pi} \alpha\left(u(\beta_1-\beta_2+\theta),u(0)\right) |\sin(\beta_1-\beta_2+\theta)|\\
&\quad\times r(E,u(\beta_1)) r(E,u(\beta_2))
d\beta_1 d\beta_2\nn\\
& = \frac{1}{2\pi} \int\limits_0^{2\pi}\int\limits_0^{2\pi} \big(\mathbf{1}\{\beta_1\in[0,\pi]\} \beta_1 \sin(\beta_1) - \mathbf{1}\{\beta_1\in(\pi,2\pi]\}(2\pi-\beta_1) \sin(\beta_1)\big) \\
&\quad \times r(E, u(\beta_1+\beta_2-\theta)) r(E, u(\beta_2))
d\beta_1 d\beta_2\nn\\
& = \frac{1}{2\pi} \int\limits_0^{2\pi}\int\limits_0^{\pi} \beta_1 \sin(\beta_1) \left[r(E, u(\beta_1+\beta_2-\theta))+ r(E, u(-\beta_1+\beta_2-\theta))  \right]\\
&\quad\times r(E,u(\beta_2))d\beta_1 d\beta_2,
\end{align*}
and hence
\begin{align}
 ^0\overline{V}_{1,1}(X,X)
& = \gamma^2 c(\alpha)^2 \int\limits_{0}^{2\pi}\int\limits_{0}^{2\pi}
\ ^0V_{1,1}(\vartheta(\theta_1)E, E)\, |\cos(\theta_1+\theta_2)|^{\alpha} |\cos(\theta_2)|^{\alpha}\, d\theta_1 d\theta_2\nn\\
& = \gamma^2 c(\alpha)^2 \int\limits_{0}^{2\pi}\int\limits_{0}^{2\pi}
 \frac{1}{2\pi} \int\limits_0^{2\pi}\int\limits_0^{\pi} \beta_1 \sin(\beta_1) r(E, u(\beta_2))\big[r(E, u(\beta_1+\beta_2-\theta_1))\nn\\
& \quad + r(E, u(-\beta_1+\beta_2-\theta_1))  \big]
d\beta_1 d\beta_2 \, |\cos(\theta_1+\theta_2)|^{\alpha} |\cos(\theta_2)|^{\alpha}\, d\theta_1 d\theta_2\nn\\
& = \frac{\gamma^2 c(\alpha)^2}{2\pi} \int\limits_{0}^{2\pi}\int\limits_{0}^{2\pi}
\int\limits_0^{2\pi}\int\limits_0^{\pi} \beta_1 \sin(\beta_1)r(E, u(\beta_2)) \big[r(E, u(\beta_1+\beta_2-\theta_1+\theta_2))\label{MixedDensSmoothGrain}\\
& \quad + r(E, u(-\beta_1+\beta_2-\theta_1+\theta_2))  \big] |\cos(\theta_1)|^{\alpha} |\cos(\theta_2)|^{\alpha}
d\beta_1 d\beta_2\,\, d\theta_1 d\theta_2.\nn
\end{align}

\end{remark}

\subsection{Planar Boolean Model with Smooth Grains}\label{BMSmoothGrains}
In this subsection we consider a Boolean model $Z$ with a grain distribution $\Q$ which is concentrated on $\cK_0\cap C_+^2.$ Then we obtain from \cite[(4.2.20)]{Schneider.1993} and Fubini's theorem with the notation \eqref{NotationUofPhi} that
\begin{align*}
 \overline{\Phi}_1^{\,0,s}(X)
& = \frac{1}{s! \omega_{1+s}}\gamma \int\limits_{\mathcal{K}_0}\int\limits_0^{2\pi} r(K,u(\varphi))u(\varphi)^s d\varphi\, \mathbb{Q}(dK)\\
& =\frac{1}{s! \omega_{1+s}} \gamma \int\limits_0^{2\pi} \int\limits_{\mathcal{K}_0} r(K,u(\varphi))\mathbb{Q}(dK)u(\varphi)^s  d\varphi,
\end{align*}
where $r(K,u)$ is the radius of curvature of $K$ at $u$, for $K\in \mathcal{K}_0\cap C_+^2 $ and $u\in S^1$, compare \cite[(2.5.22)]{Schneider.1993}.

The surface tensor mean values are now related to the Fourier coefficients of the function
$g: [0,2\pi] \rightarrow  [0,\infty)$, where
\[g(\varphi):=   \gamma  \int\limits_{\mathcal{K}_0} r(K,u(\varphi))\mathbb{Q}(dK).
\]
We denote the $s$th Fourier coefficient of $g$ by $\hat{g}(s)$. Then, we obtain for $s\in\mathbb{N}_0$ that
\begin{align*}
 \hat{g}(s) & = \frac{1}{2\pi}\int\limits_{0}^{2\pi}g(\varphi) \mathrm{e}^{-\mathrm{i}s\varphi} d\varphi =
 \frac{1}{2\pi}\int\limits_{0}^{2\pi}g(\varphi) (\cos(\varphi) - \mathrm{i} \sin(\varphi))^s  d\varphi\\
 &= \sum\limits_{j=0}^s \binom{s}{j} (-\mathrm{i})^{s-j}
 \frac{1}{2\pi}\int\limits_{0}^{2\pi}g(\varphi) (\cos\varphi)^j (\sin\varphi)^{s-j} d\varphi\\
&  = \sum\limits_{j=0}^s \binom{s}{j} (-\mathrm{i})^{s-j}
 \;\frac{s!\omega_{1+s}}{2\pi}\; \left(\overline{\Phi}_1^{\,0,s}(X)\right)_{\underbrace{1,\ldots,1}_{j \text{ times }},\underbrace{2,\ldots,2}_{s-j \text{ times }}}
\end{align*}
and in the same way that
\[
 \hat{g}(-s) = \sum\limits_{j=0}^s \binom{s}{j} \mathrm{i}^{s-j}
 \;\frac{s!\omega_{1+s}}{2\pi}\; \left(\overline{\Phi}_1^{\,0,s}(X)\right)_{\underbrace{1,\ldots,1}_{j \text{ times }},\underbrace{2,\ldots,2}_{s-j \text{ times }}}.
\]
By the theorem of Carleson \cite{Carleson.1966}, it holds
\[
\lim\limits_{N\rightarrow\infty} \sum\limits_{s=-N}^{s=N}  \hat{g}(s) \mathrm{e}^{\mathrm{i}s \varphi}= g(\varphi)
\]
for almost all $\varphi \in[0,2\pi]$.
Hence, it follows that the tensors
\[
\overline{\Phi}_1^{\,0,s}(X),\quad s\in\mathbb{N}_0,
\]
determine
\[
\gamma \; \mathbb{E}[r(Z_0,u(\varphi)]\quad \text{ for almost all } \varphi \in[0,2\pi],
\]
where $Z_0$ denotes the typical grain, i.e., a random convex body with distribution $\mathbb{Q}$.
\begin{remark}\label{ausblick}
The situation in higher dimensions is similar. Instead of just one radius of curvature one can use the product of all principal radii of curvature and the Fourier expansion can be replaced by an expansion into spherical harmonics.
\end{remark}

\section{Simulations of non-isotropic Boolean models}\label{SectionSimulations}

In this section the Boolean model $Z_{\alpha, \gamma ,E}$ introduced in the previous section is simulated within the unit square with the base grain $E$ being an ellipse or a rectangle with its main axis parallel to the first coordinate axis. Subsequently we simply write $Z$ instead of $Z_{\alpha,\gamma,E}$. The number of grains with their center located in the unit square is Poisson distributed with parameter $\gamma$. The expected occupied area
fraction is abbreviated by
\begin{align*}
  \phi := \overline{V}_2(Z)=1-\mathrm{e}^{-\overline{V}_2(X)}\, .
\end{align*}
The coordinates of the grain centers are random numbers uniformly
distributed on the unit square. The
ellipses are triangulated with 30 points. The boundary conditions are
periodic. The Computational Geometry Algorithms Library (CGAL), see \cite{CGAL},  computes the union of the triangulated
grains. Papaya calculates the Minkowski
tensors and the Euler characteristic of the triangulated Boolean
model~\cite{SCHRODERTURK.2010}.\footnote{Free Software can be
  found at\\
  \url{http://www.theorie1.physik.uni-erlangen.de/research/papaya/index.html}}
Observe that Papaya uses the different normalization \[W_j^{r,s}:= \frac{r!s!\omega_{j+s}}{n \binom{n-1}{j-1}} \Phi_{n-j}^{r,s}, \quad 0< j\leq n,\, r,s\in\N_0,\, n\geq 2,\]
for the Minkowski tensors.
The length of the main semi axis of an ellipse is $p=1/20$; the length of
the minor semi axis varies from $q=1/80$ to $1/20$. For rectangles larger
systems are accessible with a length of the main semi axis $p=1/100$ and of
the minor semi axis from $q=1/400$ to $1/100$.

\subsection{\texorpdfstring{Surface tensor density $\overline{\Phi}_1^{\,0,2}(Z)$}{Surface tensor density}}

\begin{figure}[p]
  \centering
  \subfigure[][]{%
\ifthenelse{\boolean{blackwhite}}{\includegraphics[width=0.45\textwidth]{bw_ellipse_alpha_0}}{\includegraphics[width=0.45\textwidth]{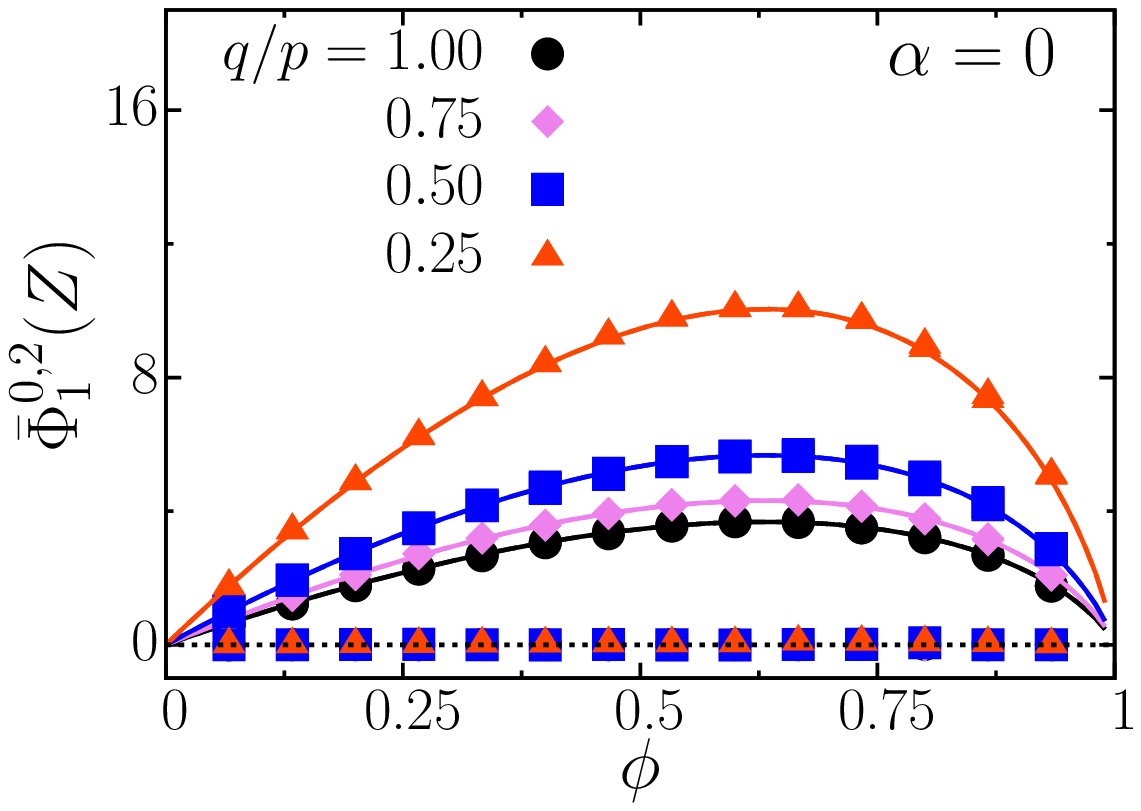}}
  }%
  \hspace{8pt}%
  \subfigure[][]{%
 \ifthenelse{\boolean{blackwhite}}{\includegraphics[width=0.45\textwidth]{bw_ellipse_alpha_1}}{\includegraphics[width=0.45\textwidth]{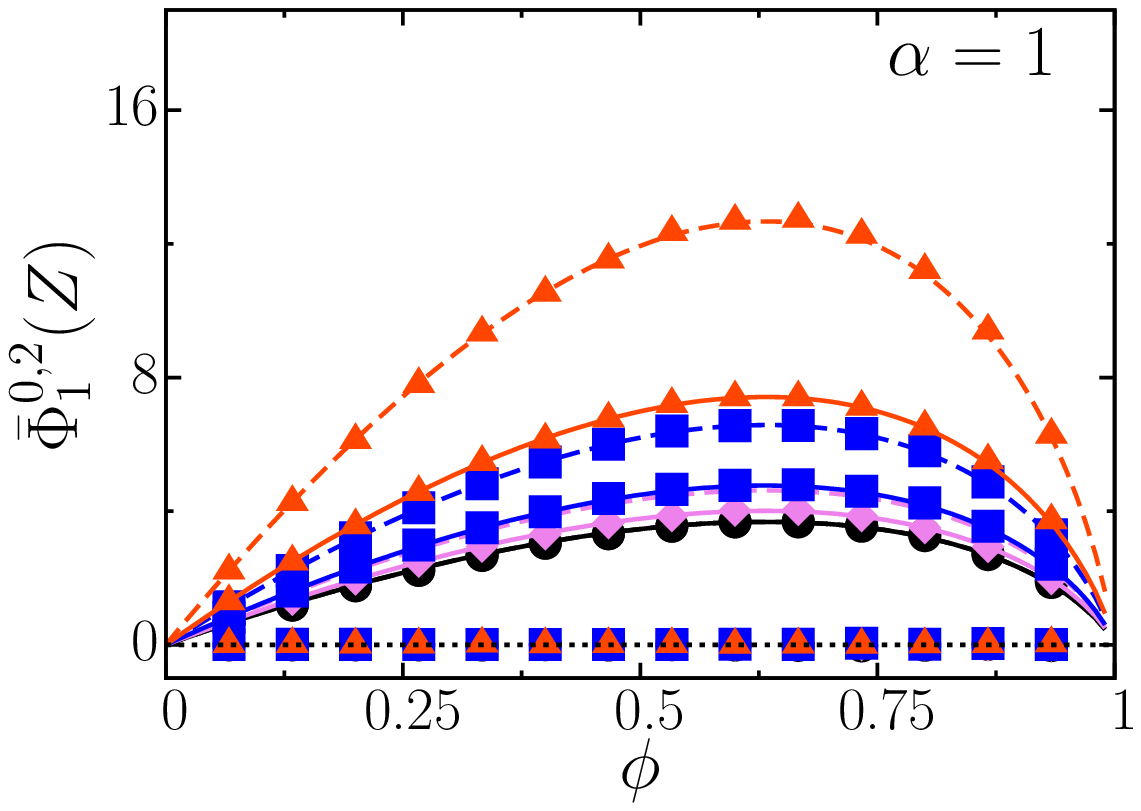}}

  }\\
  \subfigure[][]{%
  \ifthenelse{\boolean{blackwhite}}{\includegraphics[width=0.45\textwidth]{bw_ellipse_alpha_3}}{\includegraphics[width=0.45\textwidth]{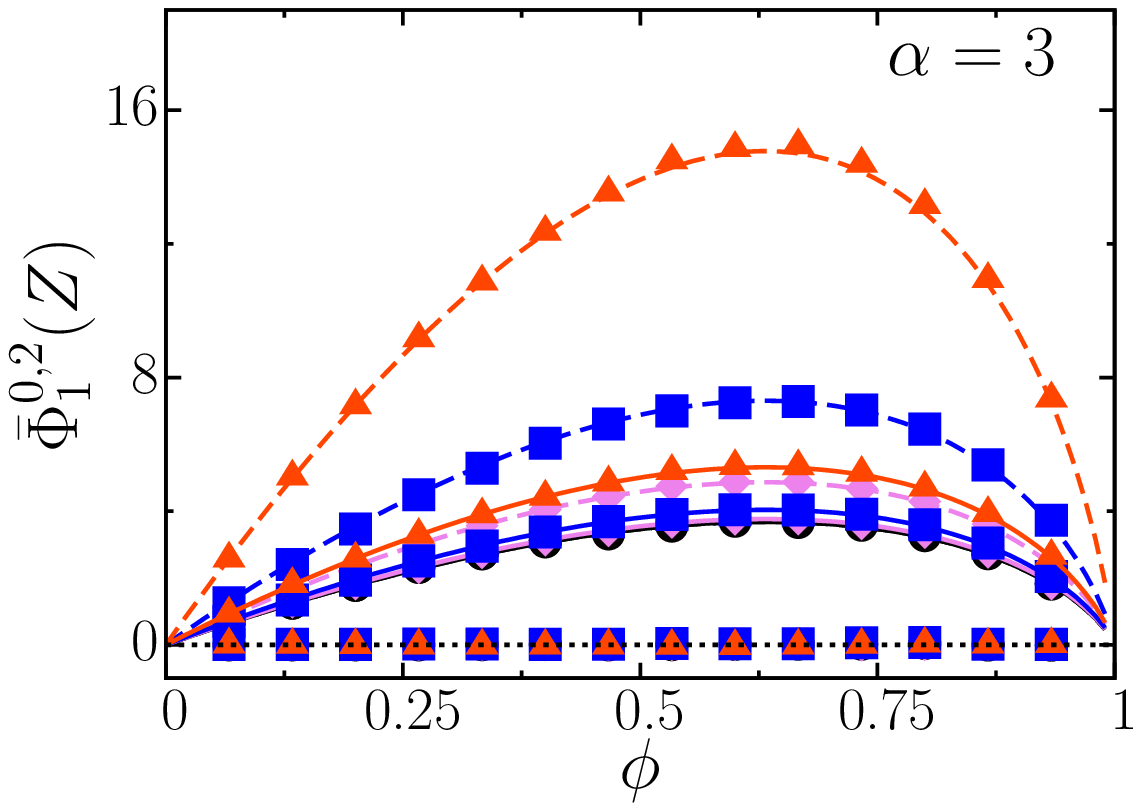}}

  }%
  \hspace{8pt}%
  \subfigure[][]{%
  \ifthenelse{\boolean{blackwhite}}{\includegraphics[width=0.45\textwidth]{bw_ellipse_alpha_25}}{\includegraphics[width=0.45\textwidth]{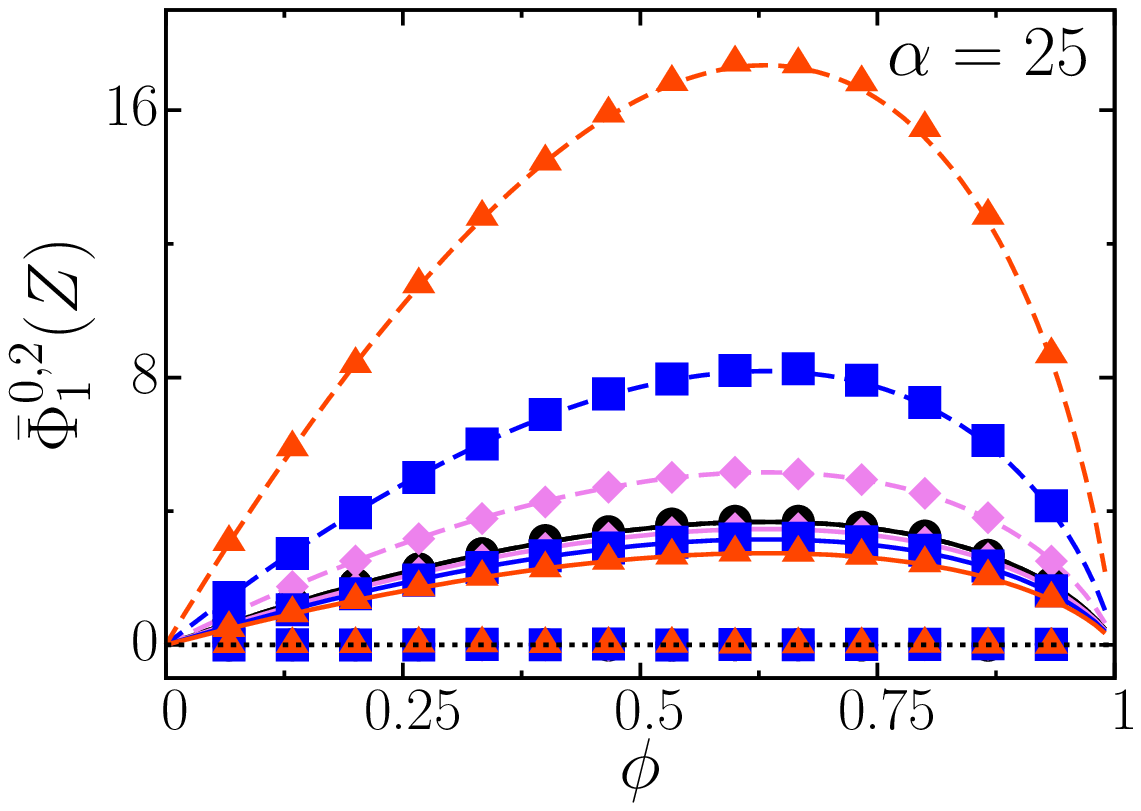}}

  }\\
  \subfigure[][]{%
  \ifthenelse{\boolean{blackwhite}}{\includegraphics[width=0.45\textwidth]{bw_ellipse_alpha_infty}}{\includegraphics[width=0.45\textwidth]{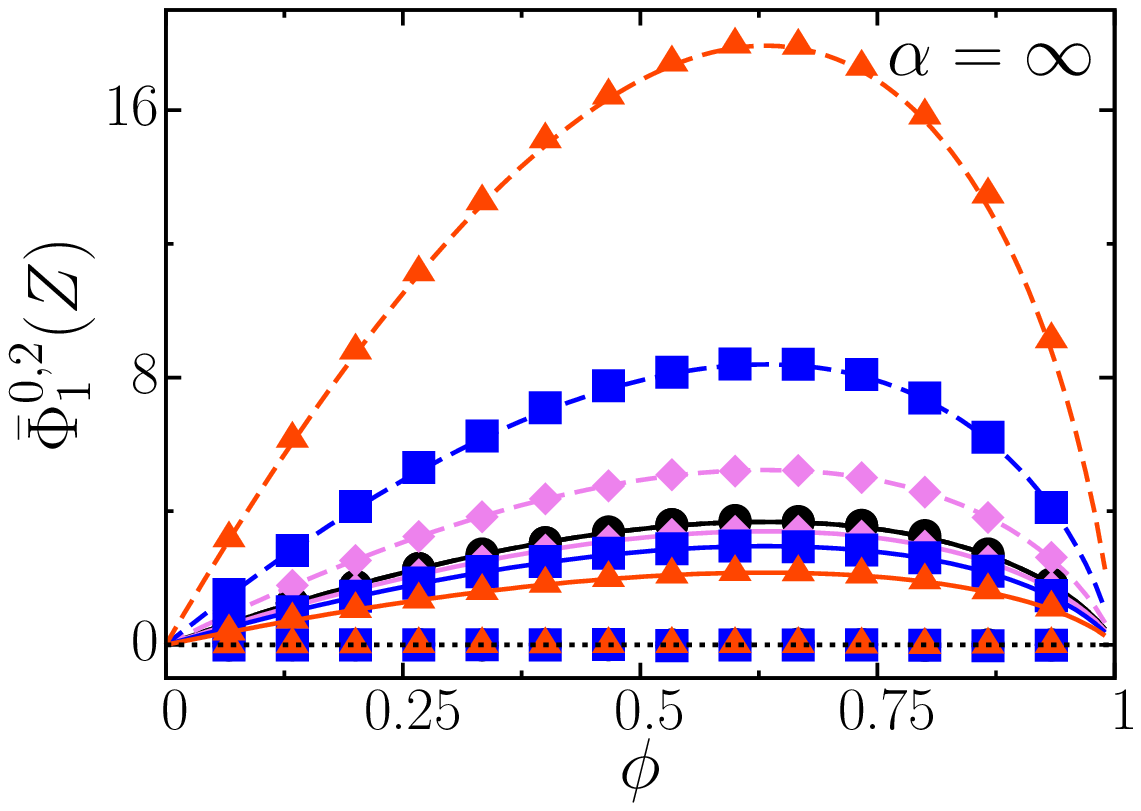}}

  }%
  \caption{Minkowski tensor density $\overline{\Phi}_1^{\,0,2}(Z)$ for the
    Boolean model with ellipses as a function of the expected occupied area
    fraction $\phi$ for varying aspect ratio $q/p$.  The numerical
    values (represented by small symbols) are compared with the analytic function from \eqref{FunctionSurfaceTensor}; $\left(\overline{\Phi}_1^{\,0,2}\right)_{1,1}$: dashed
    line; $\left(\overline{\Phi}_1^{\,0,2}\right)_{2,2}$: solid line;
    $\left(\overline{\Phi}_1^{\,0,2}\right)_{1,2}$: dotted line; (a)-(e) represent
    differently anisotropic orientation distributions.}
  \label{fig:w102_boolean_ellipses}
\end{figure}

\begin{figure}[p]
  \centering
  \subfigure[][]{%
  \ifthenelse{\boolean{blackwhite}}{\includegraphics[width=0.45\textwidth]{bw_rect_alpha_0}}{\includegraphics[width=0.45\textwidth]{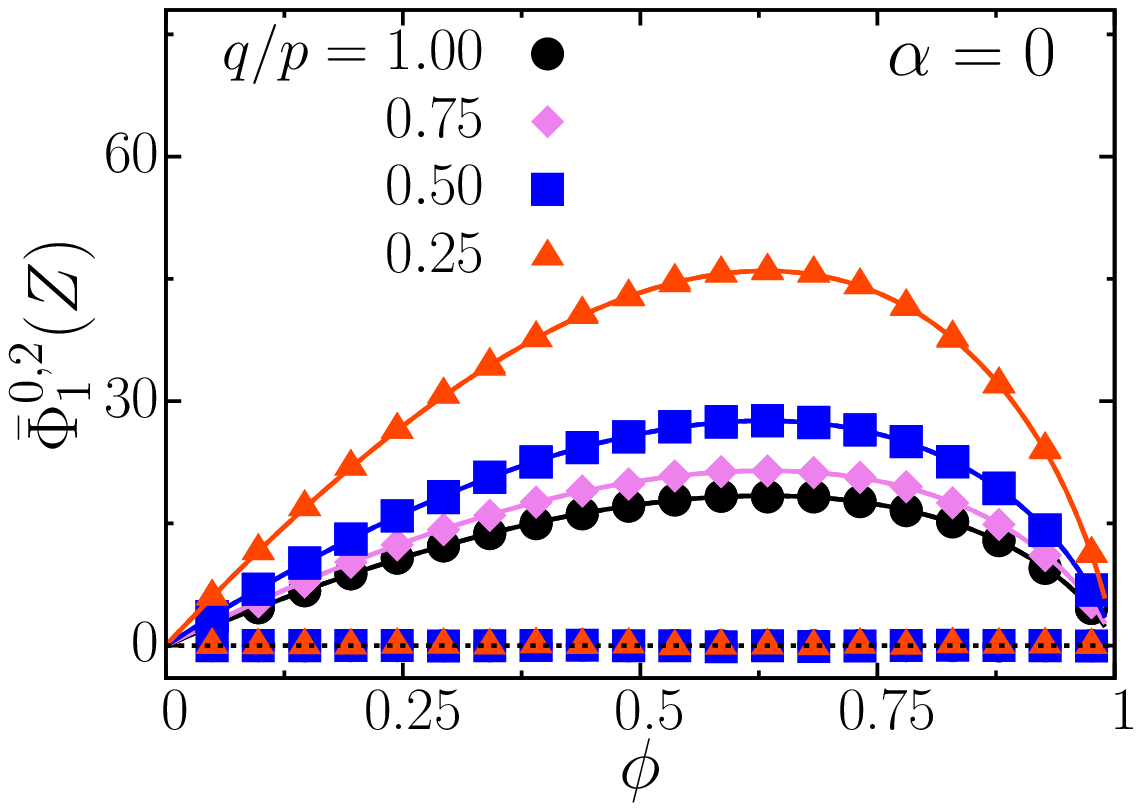}}
  }%
  \hspace{8pt}%
  \subfigure[][]{%
 \ifthenelse{\boolean{blackwhite}}{\includegraphics[width=0.45\textwidth]{bw_rect_alpha_1}}{\includegraphics[width=0.45\textwidth]{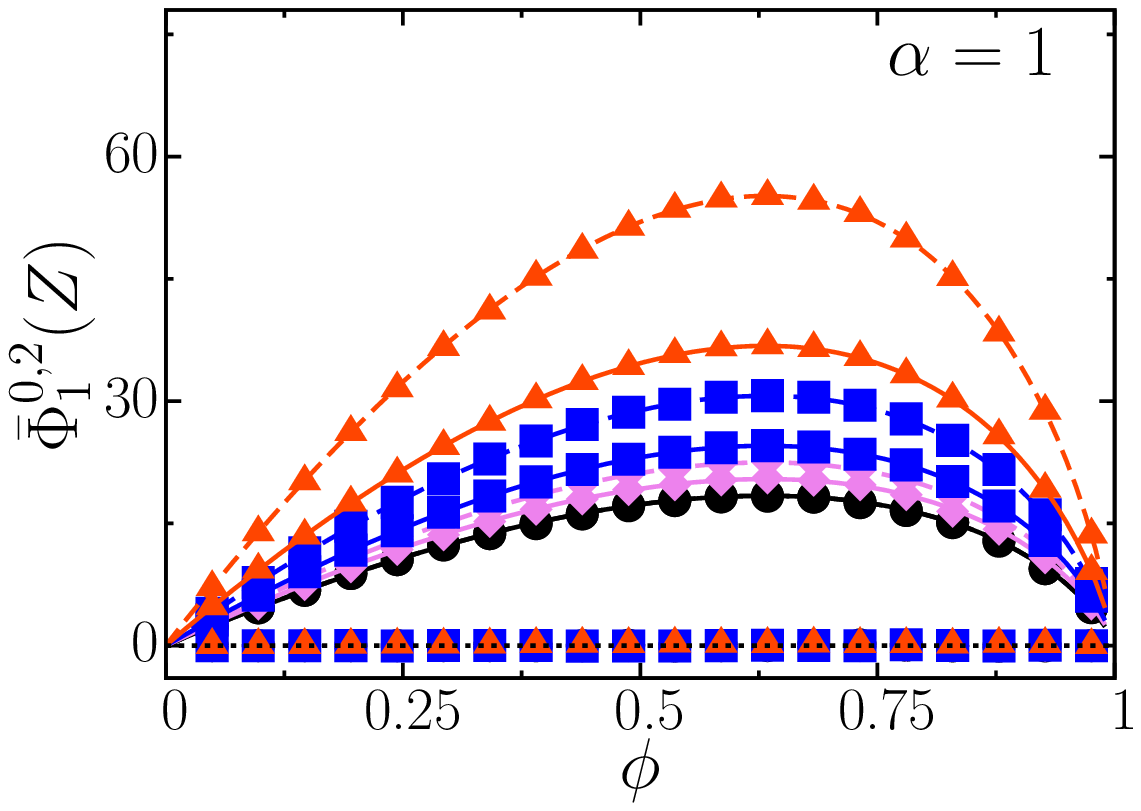}}
  }\\
  \subfigure[][]{%
    \ifthenelse{\boolean{blackwhite}}{\includegraphics[width=0.45\textwidth]{bw_rect_alpha_3}}{\includegraphics[width=0.45\textwidth]{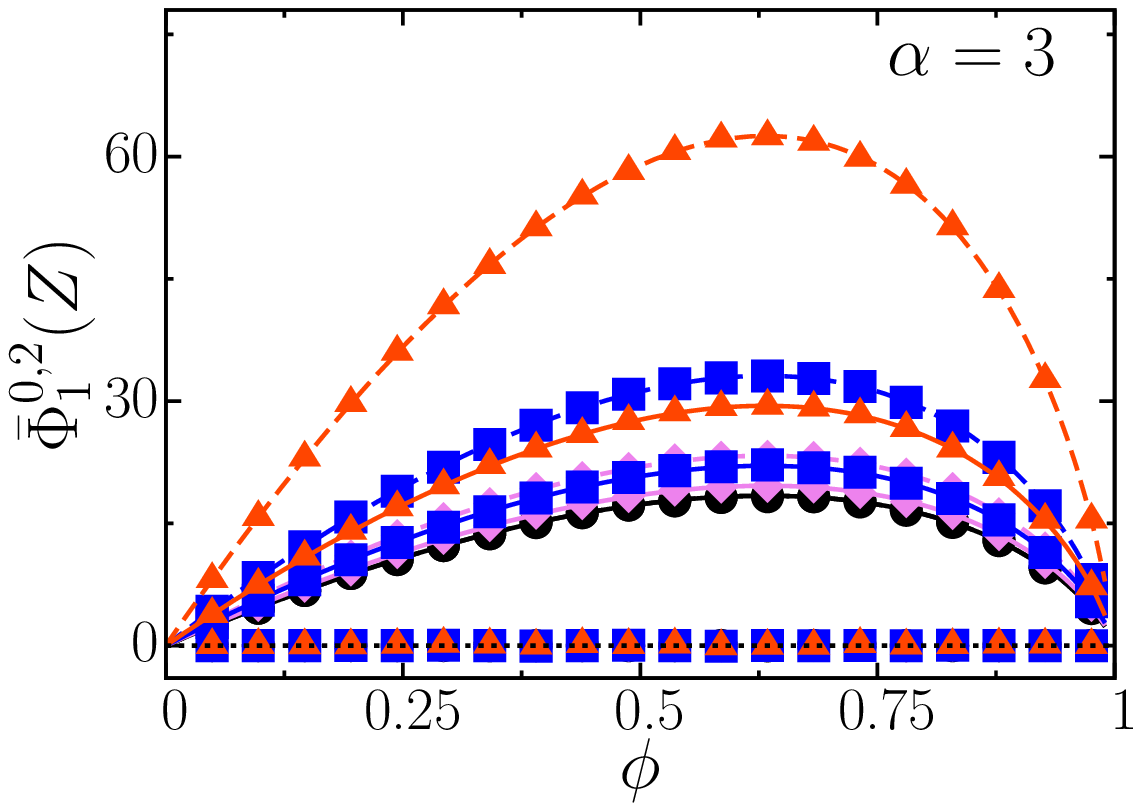}}
  }%
  \hspace{8pt}%
  \subfigure[][]{%
   \ifthenelse{\boolean{blackwhite}}{\includegraphics[width=0.45\textwidth]{bw_rect_alpha_25}}{\includegraphics[width=0.45\textwidth]{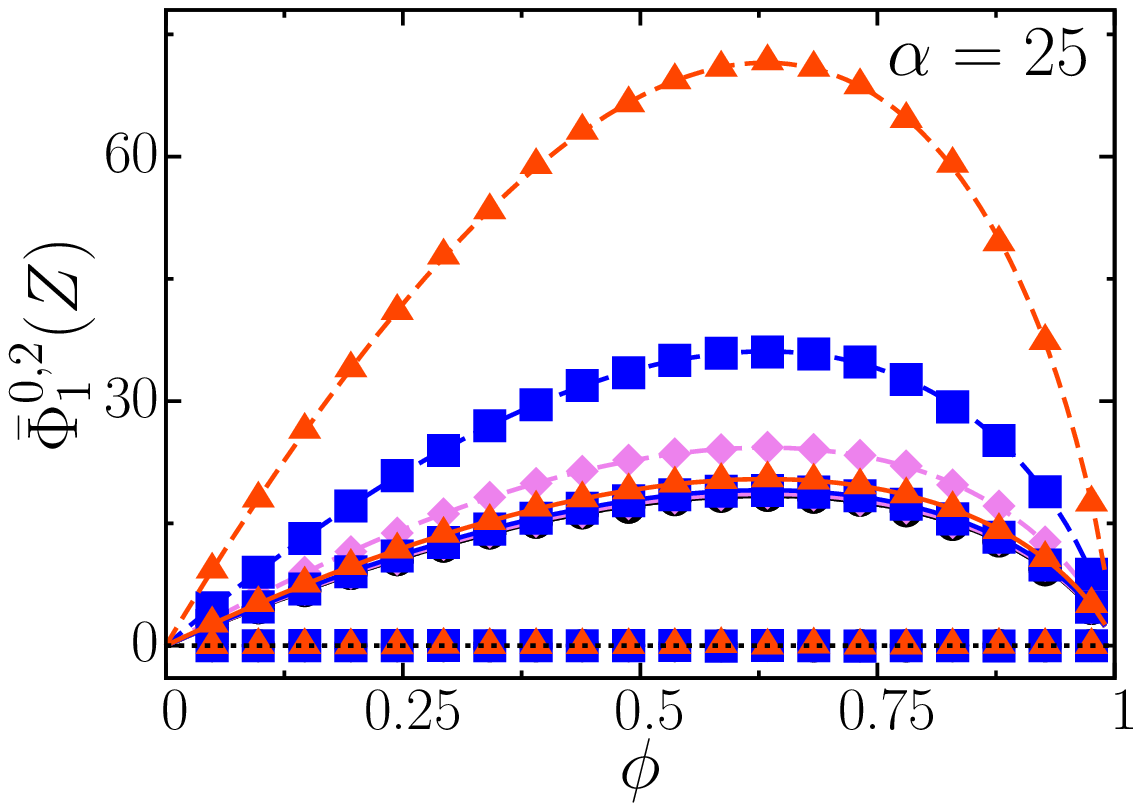}}
  }\\
  \subfigure[][]{%
 \ifthenelse{\boolean{blackwhite}}{\includegraphics[width=0.45\textwidth]{bw_rect_alpha_infty}}{\includegraphics[width=0.45\textwidth]{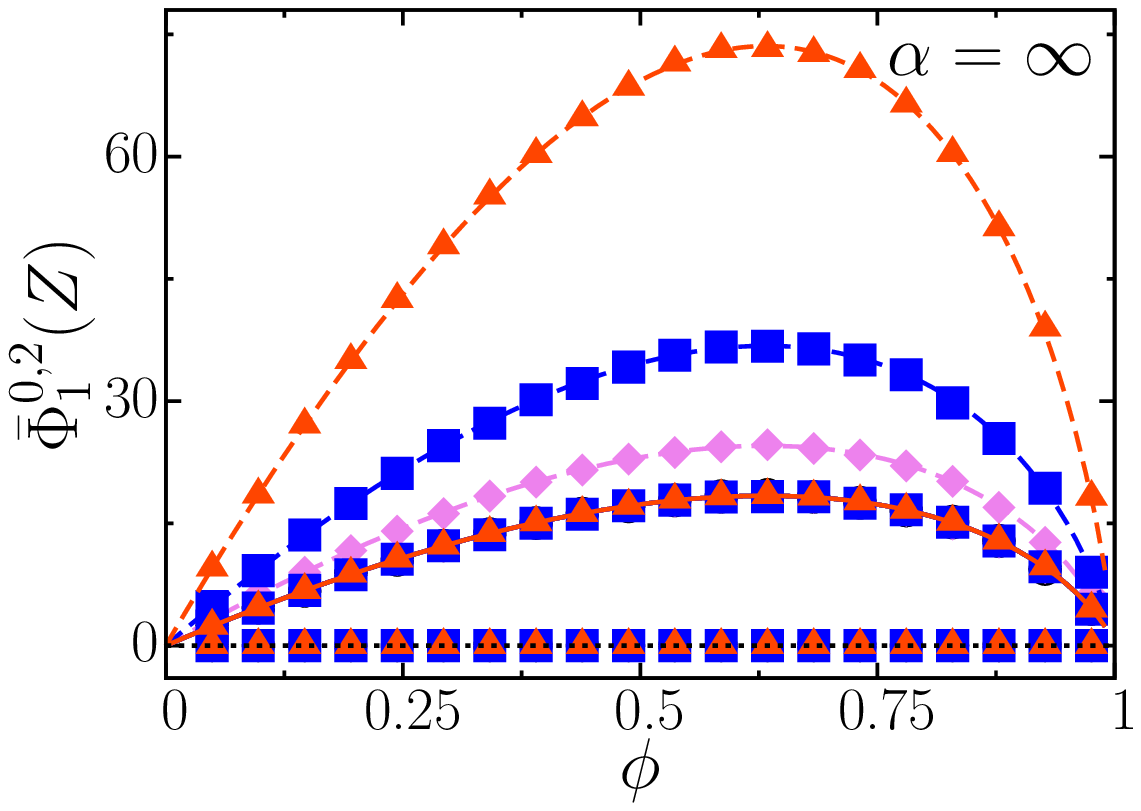}}
  }%
  \caption{Minkowski tensor density $\overline{\Phi}_1^{\,0,2}(Z)$ for the
    Boolean model with rectangles -- for details see
    Fig.~\ref{fig:w102_boolean_ellipses}.}
  \label{fig:w102_boolean_rectangles}
\end{figure}

The elements of the tensor density $\overline{\Phi}_1^{\,0,2}(Z)$ of the Boolean
model as a function of the expected occupied area fraction $\phi$ are plotted
in Figures~\ref{fig:w102_boolean_ellipses} and
\ref{fig:w102_boolean_rectangles} for different
$\alpha=0\dots\infty$. The error bars are smaller than point size. The
curves depict the analytic function
\begin{equation}\label{FunctionSurfaceTensor}
\phi \mapsto \overline{\Phi}_1^{\,0,2}(Z)= (\phi-1)\ln(1-\phi) \, c_1^{0,2}(\alpha,E),
\end{equation}
where
\[
c_1^{0,2}(\alpha,E):=\frac{1}{V_2(E)\gamma} \overline{\Phi}_1^{\,0,2}(X),
\]
which follows from Corollary \ref{Density} and \eqref{FormulaGamma}.
For given $\alpha$ the above constant $c_1^{0,2}(\alpha,E)$ can be calculated using \eqref{DensNonIsoBMs=2} and either the representation of the Minkowski tensors for polytopes (Theorem \ref{translativeMink}, {\rm(iv)}),  in the case that the base grain is a rectangle or, in the case that the base grain is an ellipse, as indicated in Remark \ref{RemarkSmoothBaseGrain} using the parametrization
\begin{equation}\label{ParamEllipse}
x: \varphi\mapsto \begin{pmatrix}\begin{array}{ll} p \cos\varphi\\ q\sin\varphi \end{array}\end{pmatrix},\quad \varphi\in[0,2\pi],
\end{equation}
of $\partial E$, where $p \geq q>0$ are the lengths of the main semi axes of the ellipse, and numerical integration. The numerical and analytic values
are in excellent agreement. Since we consider base grains which are symmetric with respect to both coordinate axes, $\Phi_1^{0,2}(E)$ has diagonal form and due to \eqref{DensNonIsoBMs=2} and \eqref{ParBMSurfaceTensRel} this property carries over to $\overline{\Phi}_1^{\,0,2}(Z)$. If the base grain $E$ is a circle or a square, then $\Phi_1^{0,2}(E)$ is proportional to the unit matrix and from \eqref{DensNonIsoBMs=2} and \eqref{ParBMSurfaceTensRel} the same follows for $\overline{\Phi}_1^{\,0,2}(Z)$.
Figure~\ref{fig:w102_of_alpha_boolean_ellipses} shows for a Boolean
model with ellipses the difference of the diagonal elements of the
tensor density $\overline{\Phi}_1^{\,0,2}(Z)$, i.e., the difference of the eigenvalues in $x$- and
$y$-direction, as a function of the orientation parameter $\alpha$,
again for different aspect ratios $q/p$. With an increasing $\alpha$ the probability density function $f_{\alpha}$ of the random angle $\theta$ is more and more concentrated around $0$ and the
difference in the eigenvalues increases, obviously except for
circles. All simulations were performed at expected occupied area fraction
$\phi=1/3$.

\begin{figure}[tb]
  \centering
  \ifthenelse{\boolean{blackwhite}}{\includegraphics[width=0.45\textwidth]{bw_ellipse_alpha_rung}}{\includegraphics[width=0.45\textwidth]{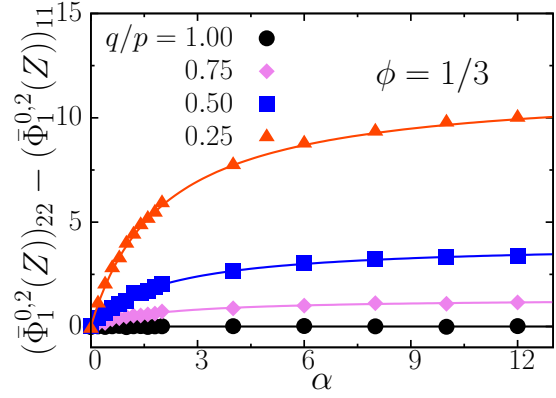}}
  \caption{Difference of the eigenvalues of the tensor density
    $\overline{\Phi}_1^{\,0,2}(Z)$ as a function of the orientation parameter
    $\alpha$ for different aspect ratios $q/p$ of the ellipses. The
    expected occupied area fraction for the simulation was chosen to be
    $\phi=1/3$. The lines show the analytic functions which follow similarly as in \eqref{FunctionSurfaceTensor}. }
  \label{fig:w102_of_alpha_boolean_ellipses}
\end{figure}

\subsection{Estimation of model parameters}\label{EstOfModelParameters}

Given a measured eigenvalue of the tensor density $\hat{\Phi}_1^{0,2}(Z)$
and the measured occupied area fraction $\hat{\phi}$ of a sample of
the Boolean model, Equations~(\ref{eq:gamma_Boolean_model}) and
(\ref{eq:phi_1_02_Boolean_model}) allow an estimate of both the
intensity $\gamma$ and the orientation parameter
$\alpha$. From equation \eqref{FormulaGamma} and \eqref{FormulaAlpha} we deduce the estimates
\begin{align}
  \hat{\gamma} &= -\frac{\ln\left( 1 - \hat{\phi} \right)}{V_2(E)}\, .
  \label{eq:gamma_Boolean_model}
\end{align}
and
\begin{align}
  \hat{\alpha} &=
  \frac{\hat{\gamma}\left(\left(\Phi_1^{0,2}(E)\right)_{11}+\left(\Phi_1^{0,2}(E)\right)_{22} \right)-2\mathrm{e}^{\hat{\gamma}V_2(E)}\left(\hat{\Phi}_1^{0,2}(Z)\right)_{11} }{\mathrm{e}^{\hat{\gamma}V_2(E)}\left(\hat{\Phi}_1^{0,2}(Z)\right)_{11} -\hat{\gamma}\left(\Phi_1^{0,2}(E)\right)_{11}}
  \, .\label{eq:phi_1_02_Boolean_model}
\end{align}
In the simulations the intensity was chosen to be $\gamma= \ln(15/14)\approx 0.06899$.
The estimate is not well defined for a base grain $E$ with
 \[
\left(\Phi_1^{0,2}(E)\right)_{11}=\left(\Phi_1^{0,2}(E)\right)_{22},
\]
which is called isotropic with respect to the tensor density $\overline{\Phi}_1^{\,0,2}$. Clearly, for
$\alpha \rightarrow \infty$ also the estimate diverges.
Notice that estimates $\hat{\alpha}<-1$ may appear, although $\alpha<-1$
is forbidden.

Figure~\ref{fig:histogram_estimate} depicts histograms
of the estimates of orientation parameter $\alpha$ and intensity
$\gamma$ for the specific choice $V_2(E)=1$; for a Boolean model with ellipses in the plots (a) and (b) and
with rectangles in (c) and (d), respectively. In each case, 1000
simulations were performed with an aspect ratio $q/p=1/4$ of the
grains. The length of the simulation box was chosen to be $L=100\, p$
for both the ellipses and the rectangles. In each plot, the \ifthenelse{\boolean{blackwhite}}{dashed}{black} line depicts the true value of the parameter which is to be estimated. The
mean of the distribution of the estimates and the error of the mean are
computed via bootstrapping; the $1 \sigma$ band of the mean is also
shown as a broad\ifthenelse{\boolean{blackwhite}}{ }{ colored }line.

\begin{figure}[tb]
  \centering
  \subfigure[][]{%
   \ifthenelse{\boolean{blackwhite}}{\includegraphics[width=0.45\textwidth]{bw_hist_ell_alpha_samp}}{\includegraphics[width=0.45\textwidth]{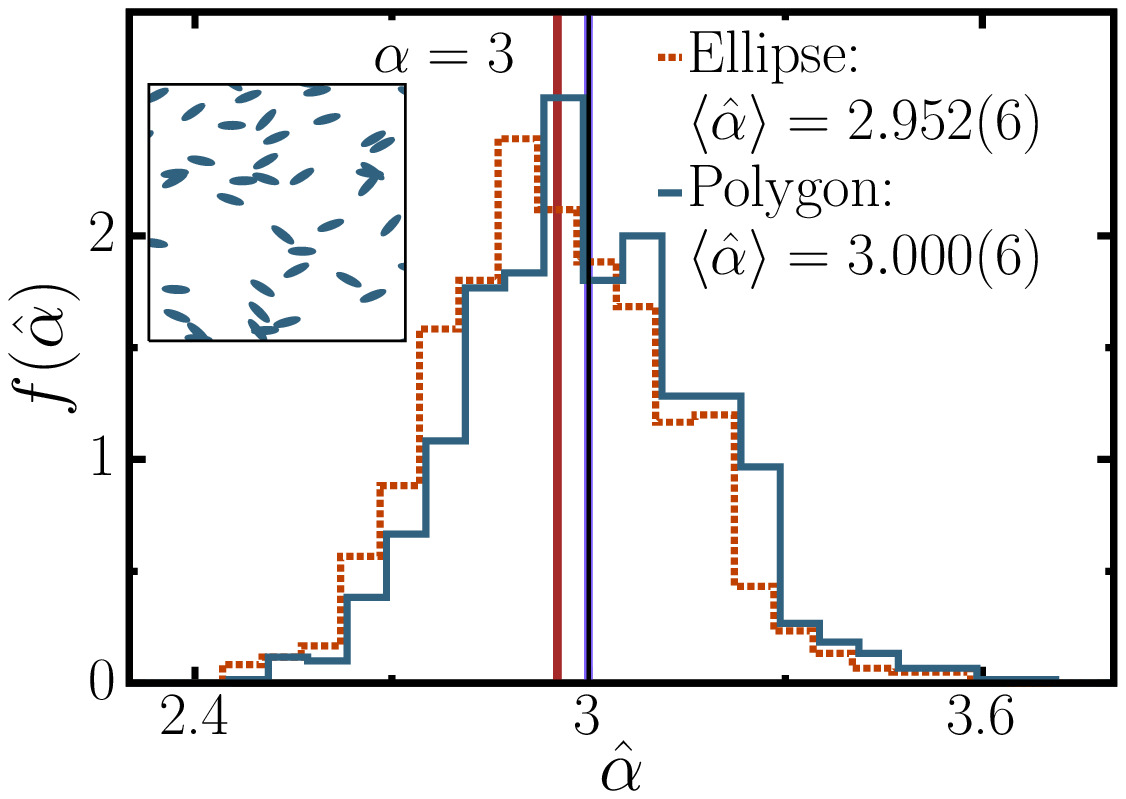}}
  }%
  \hspace{8pt}%
  \subfigure[][]{%
  \ifthenelse{\boolean{blackwhite}}{\includegraphics[width=0.45\textwidth]{bw_hist_ell_int}}{\includegraphics[width=0.45\textwidth]{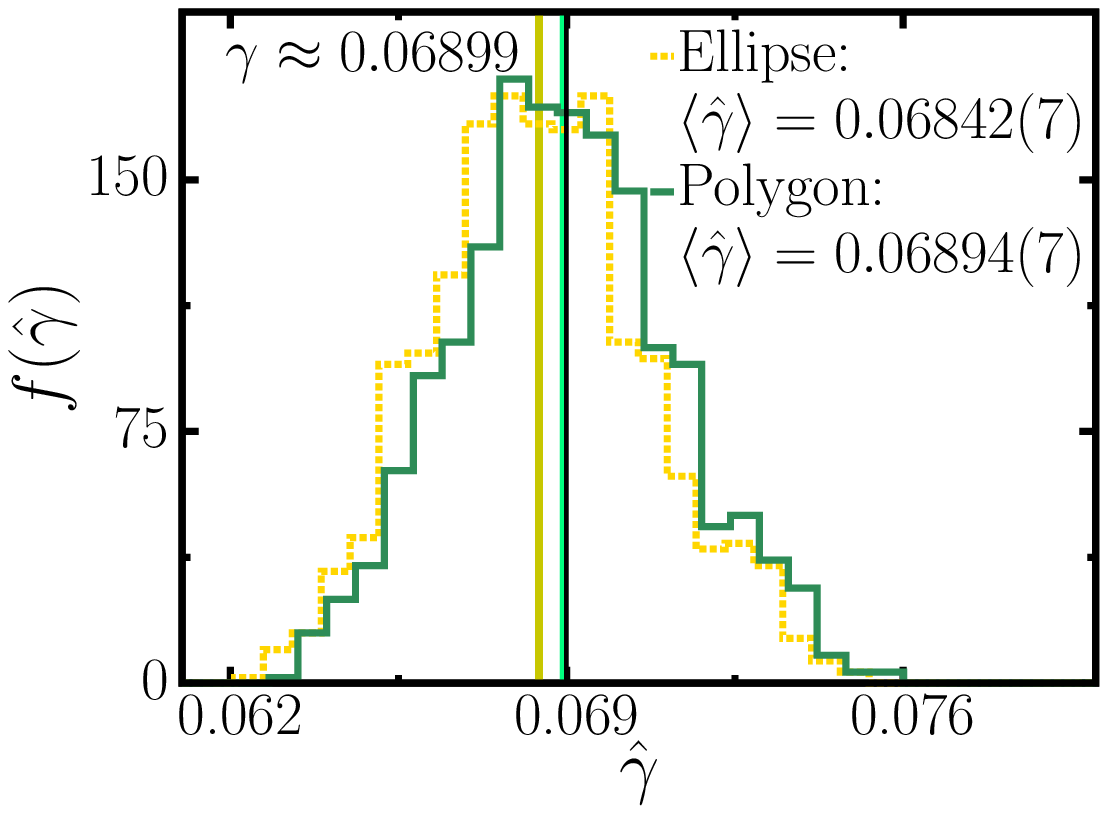}}
  }\\
  \subfigure[][]{%
    \ifthenelse{\boolean{blackwhite}}{\includegraphics[width=0.45\textwidth]{bw_hist_rect_alpha_samp}}{\includegraphics[width=0.45\textwidth]{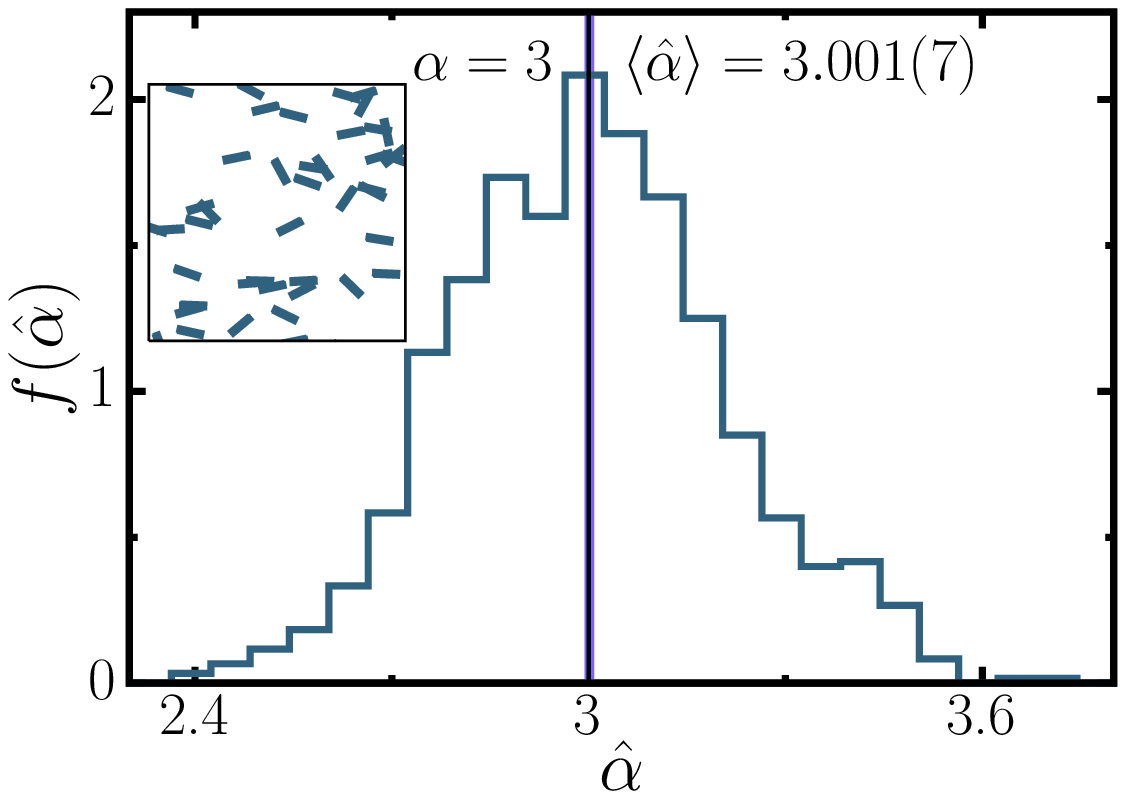}}
  }%
  \hspace{8pt}%
  \subfigure[][]{%
  \ifthenelse{\boolean{blackwhite}}{\includegraphics[width=0.45\textwidth]{bw_hist_rect_int}}{\includegraphics[width=0.45\textwidth]{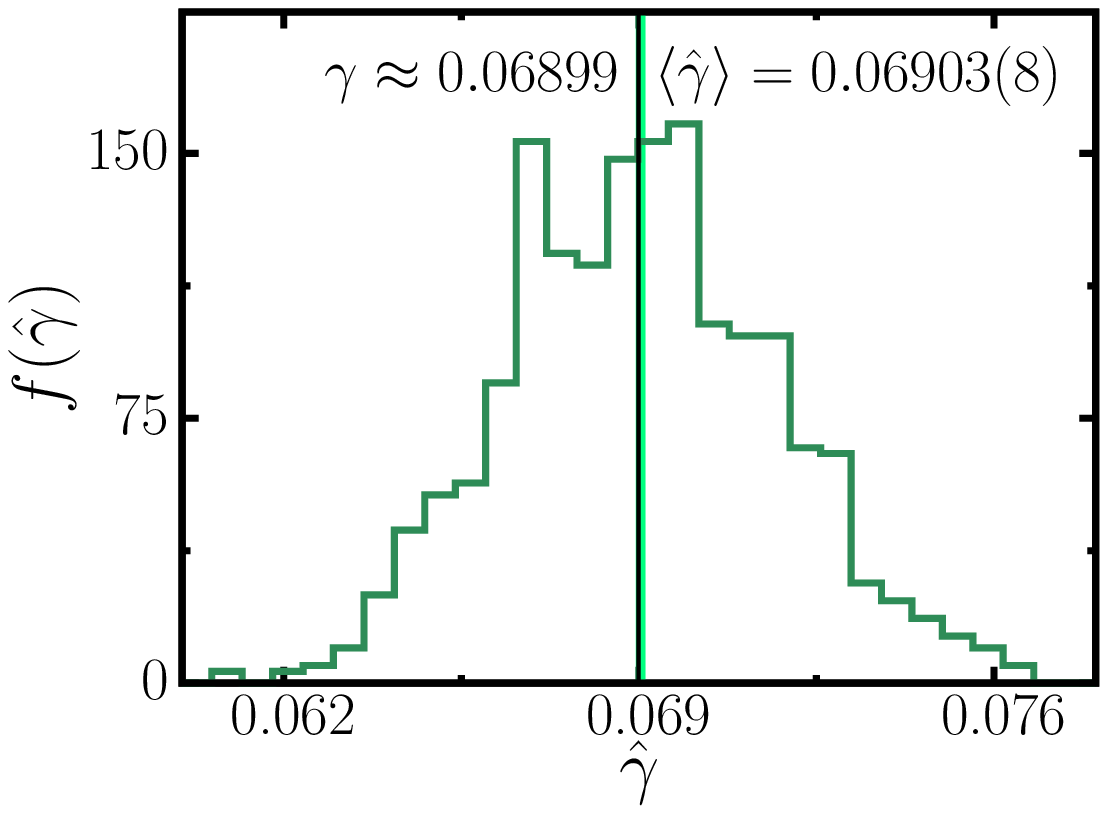}}
  }

  \caption[toc entry]{Histograms of the estimates of orientation parameter $\alpha$ and
intensity $\gamma$ for the specific choice $V_2(E)=1$; (a) and (b) for ellipses; (c) and (d) for
rectangles. The\ifthenelse{\boolean{blackwhite}}{ dashed }{ black }lines depict the true values of the parameters
which are to be estimated. The broad\ifthenelse{\boolean{blackwhite}}{ }{ colored }lines show the $1\sigma$
band of the mean of the estimate. For (a) and (b) the parameters were
estimated both with the grain characteristics of an ellipse (dotted
lines) and of the polygon which was actually used for the simulation
(solid line). The insets in (a) and (c) illustrate samples of the
Boolean models.}
  \label{fig:histogram_estimate}
\end{figure}
Both the estimation of the intensity and of the orientation parameter
of the Boolean model with rectangles in Figs.~(c) and (d) appear to be
unbiased. The original parameters can be retained with high
statistical precision. However, for smaller system sizes finite size
effects may lead to a significant bias. For the simulation of the
Boolean model with ellipses in Figs.~(a) and (b) a polygon with 30
vertices approximated an ellipse; the relative error in the area is
$0.7\,\%$ and in $\left(\Phi_1^{0,2}(E)\right)_{1,1}$ only
$0.2\,\%$. Nevertheless, when the parameters of the Boolean model were
estimated using the single grain characteristics of an ellipse, the
mean was eight standard deviations away from the true value. However,
when the grain characteristics of the polygon were used, the estimator was bias free within
statistical significance. The method can be used as a very sensitive
test for the parameters of the system.

\begin{figure}[p]
  \centering
  \subfigure[][]{%
  \ifthenelse{\boolean{blackwhite}}{\includegraphics[width=0.45\textwidth]{bw_ec_ellipse_alpha_0}}{\includegraphics[width=0.45\textwidth]{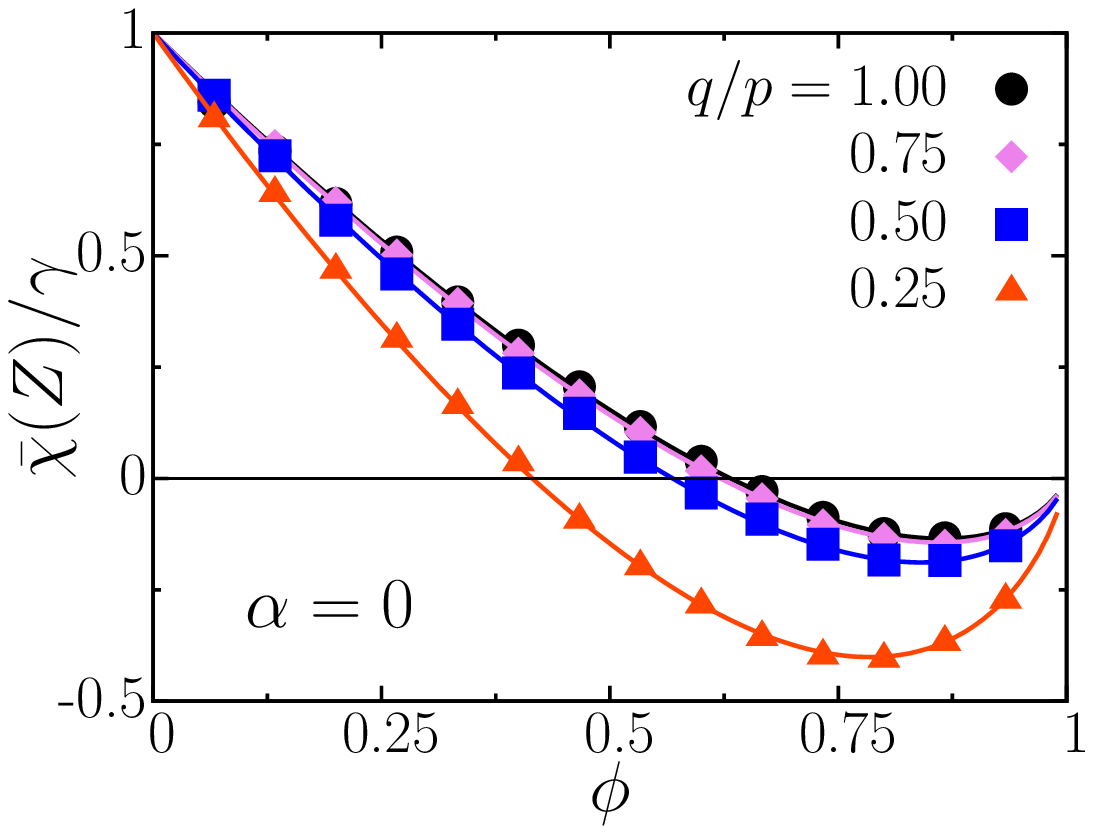}}
  }%
  \hspace{8pt}%
  \subfigure[][]{%
  \ifthenelse{\boolean{blackwhite}}{\includegraphics[width=0.45\textwidth]{bw_ec_ellipse_alpha_1}}{\includegraphics[width=0.45\textwidth]{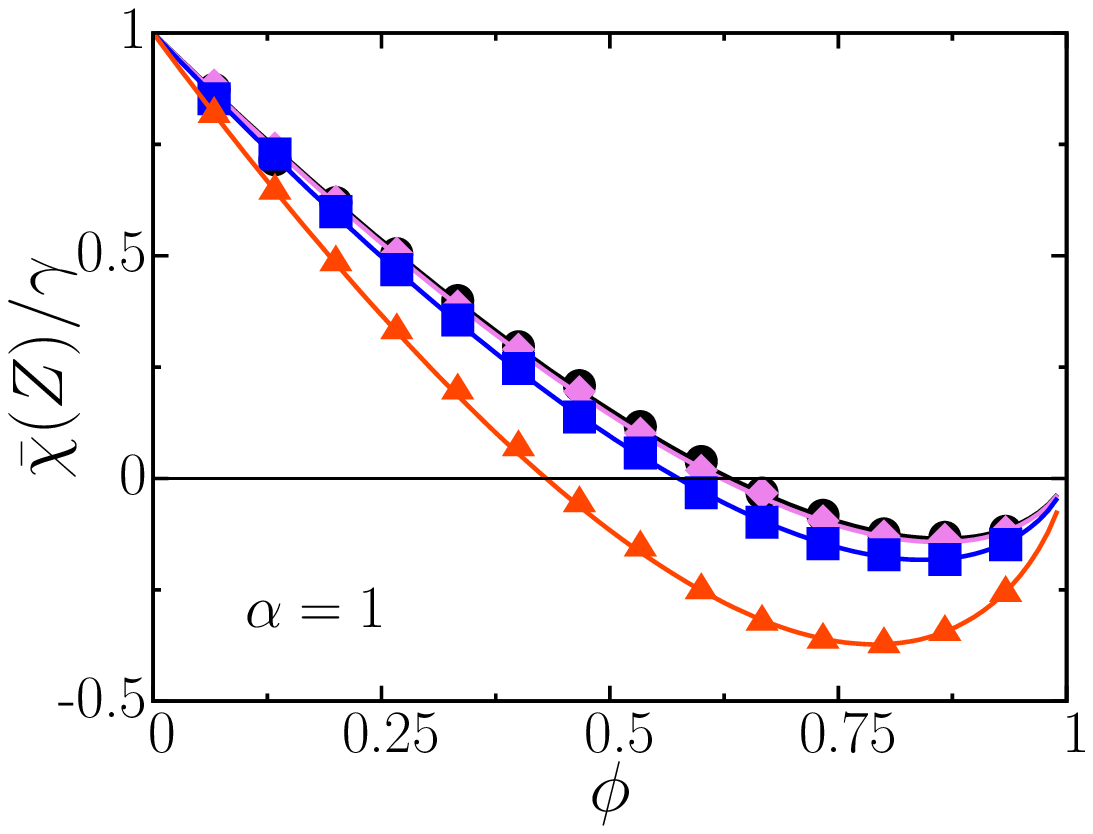}}
  }\\
  \subfigure[][]{%
    \ifthenelse{\boolean{blackwhite}}{\includegraphics[width=0.45\textwidth]{bw_ec_ellipse_alpha_3}}{\includegraphics[width=0.45\textwidth]{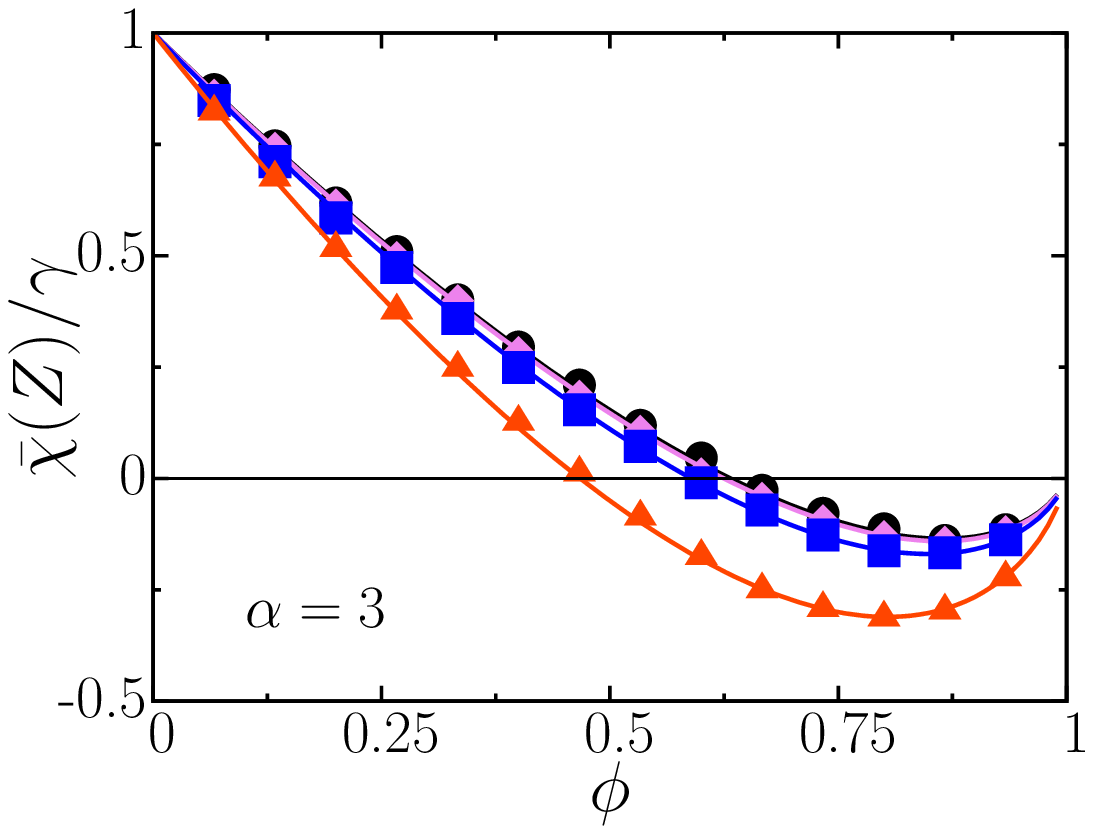}}
  }%
  \hspace{8pt}%
  \subfigure[][]{%
    \ifthenelse{\boolean{blackwhite}}{\includegraphics[width=0.45\textwidth]{bw_ec_ellipse_alpha_25}}{\includegraphics[width=0.45\textwidth]{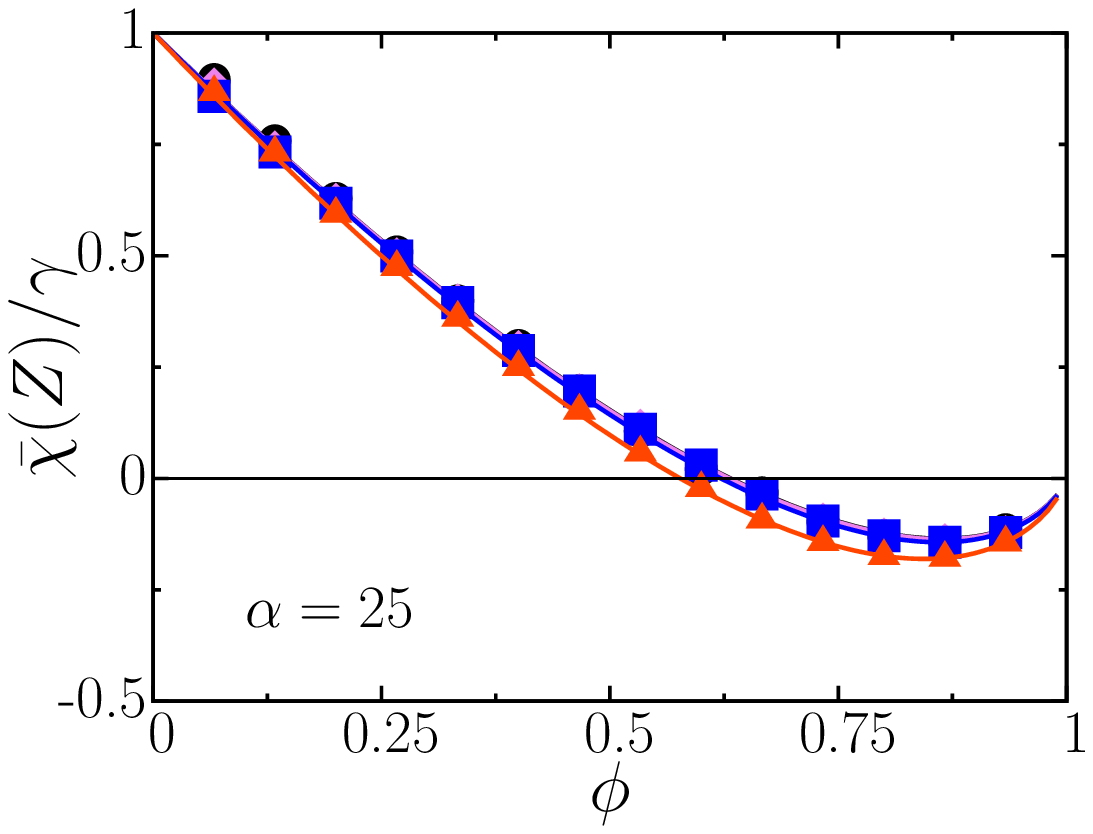}}
  }\\
  \subfigure[][]{%
    \ifthenelse{\boolean{blackwhite}}{\includegraphics[width=0.45\textwidth]{bw_ec_ellipse_alpha_infty}}{\includegraphics[width=0.45\textwidth]{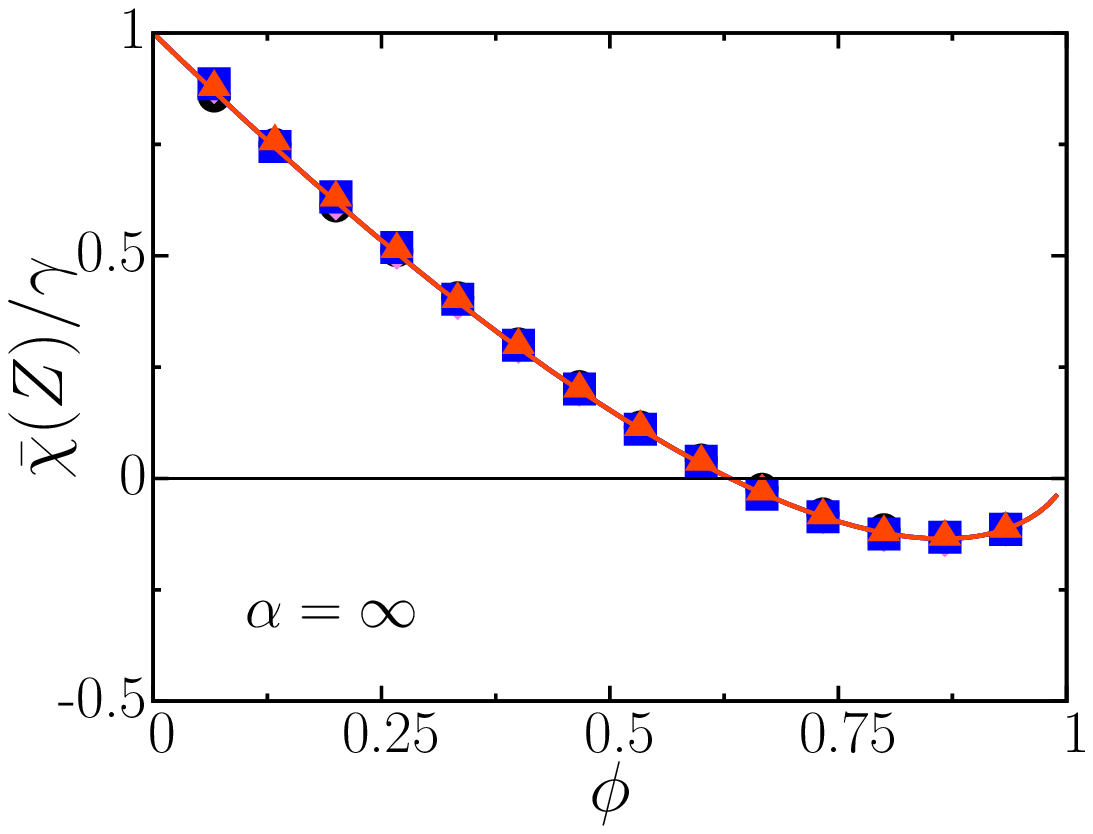}}
   }%
  \caption{Euler characteristic density $\overline{\chi}(Z)$ normalized with
    the intensity $\gamma$ for the Boolean model with ellipses as a
    function of the expected occupied area fraction $\phi$ for varying aspect
    ratio $q/p$. The numerical values are
    compared with the analytic functions from
    Eq.~\eqref{FunctionEulerChar}. (a)-(e) represent
    differently anisotropic orientation distributions.}
  \label{fig:chi_boolean_ellipses}
\end{figure}

\begin{figure}[p]
  \centering
  \subfigure[][]{%
  \ifthenelse{\boolean{blackwhite}}{\includegraphics[width=0.45\textwidth]{bw_ec_rect_alpha_0}}{\includegraphics[width=0.45\textwidth]{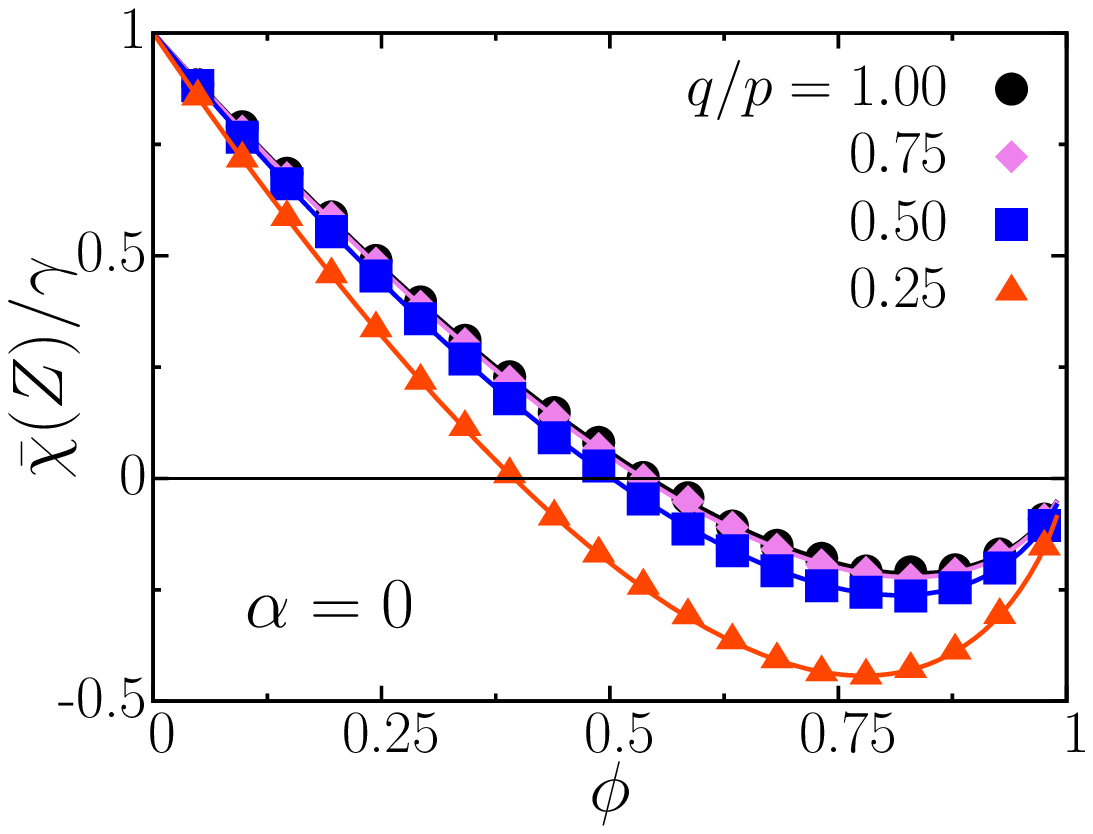}}
  }%
  \hspace{8pt}%
  \subfigure[][]{%
  \ifthenelse{\boolean{blackwhite}}{\includegraphics[width=0.45\textwidth]{bw_ec_rect_alpha_1}}{\includegraphics[width=0.45\textwidth]{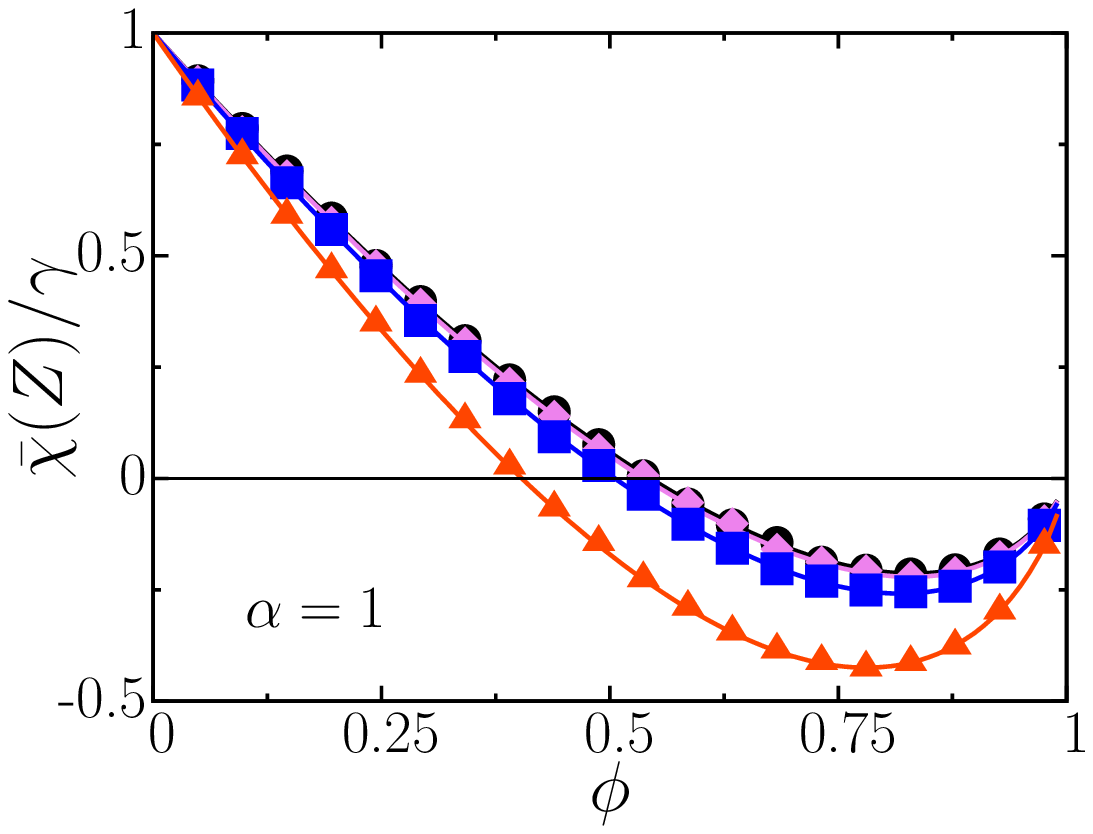}}
  }\\
  \subfigure[][]{%
  \ifthenelse{\boolean{blackwhite}}{\includegraphics[width=0.45\textwidth]{bw_ec_rect_alpha_3}}{\includegraphics[width=0.45\textwidth]{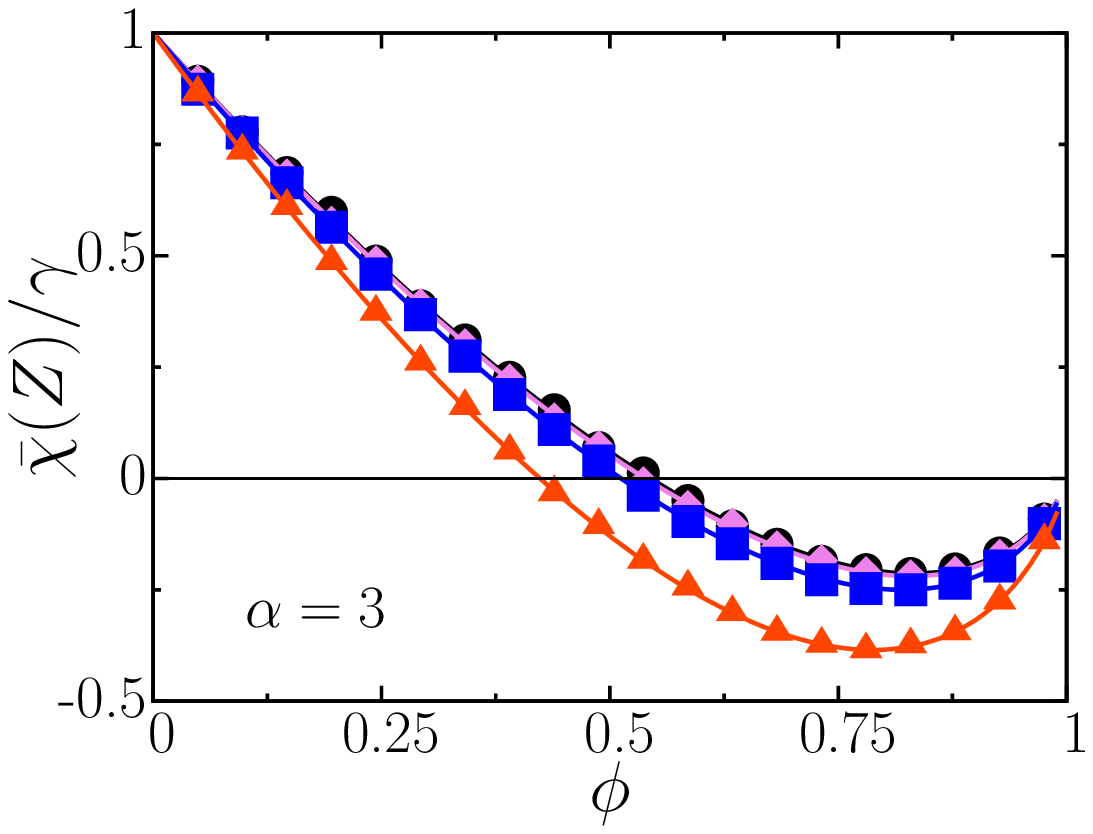}}
  }%
  \hspace{8pt}%
  \subfigure[][]{%
  \ifthenelse{\boolean{blackwhite}}{\includegraphics[width=0.45\textwidth]{bw_ec_rect_alpha_25}}{\includegraphics[width=0.45\textwidth]{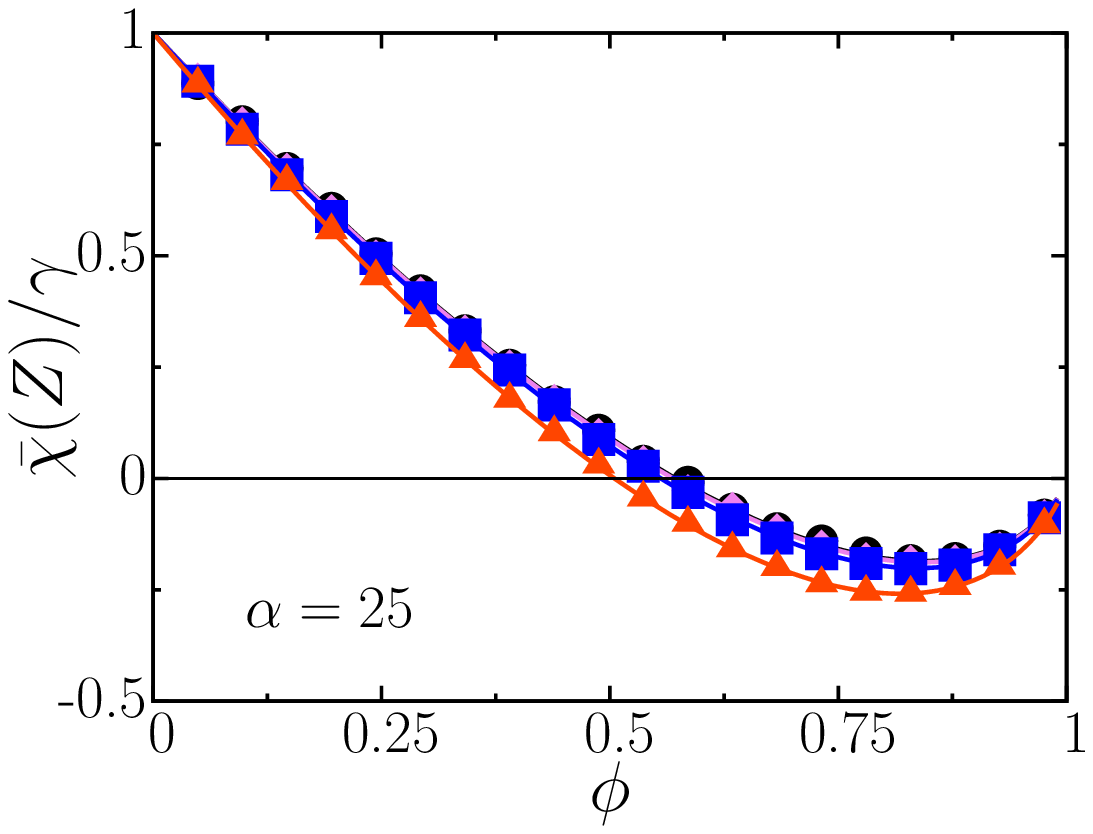}}
  }\\
  \subfigure[][]{%
  \ifthenelse{\boolean{blackwhite}}{\includegraphics[width=0.45\textwidth]{bw_ec_rect_alpha_infty}}{\includegraphics[width=0.45\textwidth]{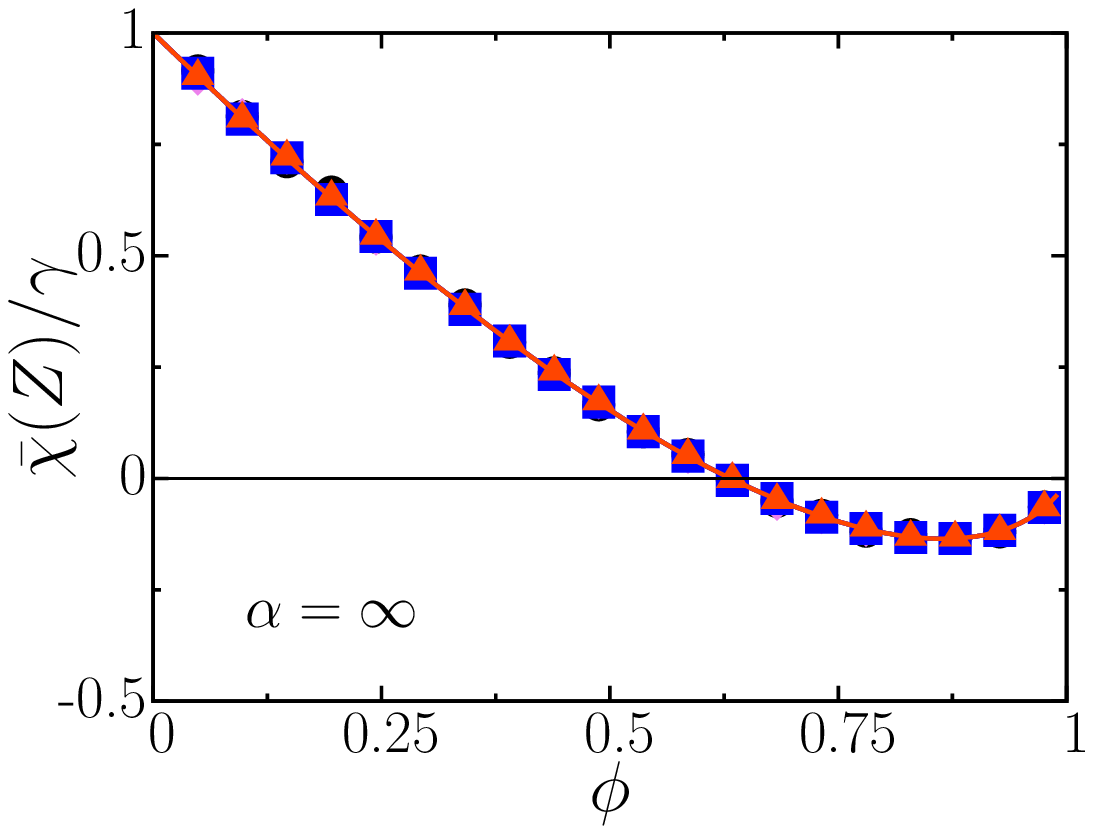}}
  }%
  \caption{Euler characteristic density $\overline{\chi}(Z)$ normalized with
    the intensity $\gamma$ for the Boolean model with rectangles; the
    lines show the analytic function from
    Eq.~\eqref{FunctionEulerChar} -- for details see
    Fig.~\ref{fig:chi_boolean_ellipses}.}
  \label{fig:chi_boolean_rectangles}
\end{figure}

\subsection{\texorpdfstring{Euler characteristic density $\overline{\chi}(Z)$}{Euler characteristic density}}

Figures~\ref{fig:chi_boolean_ellipses} and
\ref{fig:chi_boolean_rectangles} show the Euler characteristic density
$\overline{\chi}(Z)$ of anisotropic Boolean models with ellipses or
rectangles, respectively. The Euler characteristic density
$\overline{\chi}(Z)$, which is normalized with the intensity $\gamma$, is plotted as a
function of the expected occupied area fraction $\phi$. The Euler
characteristic density $\overline{\chi}$ was computed for the same samples
which were used for Figures~\ref{fig:w102_boolean_ellipses} and
\ref{fig:w102_boolean_rectangles}. The error bars are smaller than
point size. The curves depict the analytic function
\begin{equation}\label{FunctionEulerChar}
\phi \mapsto \overline{V_0}(Z)/ \gamma=
(1-\phi)[1+c_0(\alpha,E)\ln(1-\phi)],
\end{equation}
where
\[
c_0(\alpha,E):=\frac{1}{2V_2(E)\gamma^2}\ ^0\overline{V}_{1,1}(X,X),
\]
which is obtained from Corollary \ref{DensityDimTwoThree} and \eqref{FormulaGamma}.
In the case that the base grain is a rectangle the above constant $c_0(\alpha,E)$ can be calculated for given $\alpha$ using \eqref{mixedDensityGeneral} and the representation of the mixed Minkowski tensors for polytopes, see Theorem \ref{translativeMink}, {\rm(iv)}, which leads to the formula
\begin{align*}
  \ ^0V_{1,1}(R,\vartheta(\theta)R) = (a^2 + b^2)|\sin\theta| + 2ab|\cos\theta|, \quad \theta\in [0,2\pi],
\end{align*}
where $R$ is a rectangle with side lengths $a, b >0.$
In the case that the base grain is an ellipse with its boundary parametrized by \eqref{ParamEllipse}, the constant $c_0(\alpha,E)$ can be calculated using \eqref{mixedDensityGeneral}, the formula \eqref{MixedDensSmoothGrain}, that for $u\in S^1$ the curvature of $E$ at a boundary point with outer normal $u$ is
\[
r(E,u) = \frac{p^2 q^2}{(p^2 u_1^2+q^2 u_2^2)^{3/2}}
\]
and numerical integration.
The values obtained from simulations and the analytic values are in excellent
agreement. The Euler characteristic for
aligned grains is independent of the aspect ratio, because in this
case a change in the aspect ratio is simply an elongation of the
system in one direction, which does not change the
topology.
Concluding we emphasis that tensorial functionals, in particular Minkowski
 tensors, are a versatile tool to characterize orientational distributions
 in Boolean models which is important for many applications in materials
 science.





\bibliographystyle{model1b-num-names}


\providecommand{\bysame}{\leavevmode\hbox to3em{\hrulefill}\thinspace}
\providecommand{\MR}{\relax\ifhmode\unskip\space\fi MR }
\providecommand{\MRhref}[2]{%
  \href{http://www.ams.org/mathscinet-getitem?mr=#1}{#2}
}
\providecommand{\href}[2]{#2}

\bigskip
\noindent
{\em Authors' addresses:}
\vspace{.5cm}\\
\noindent

Julia H\"orrmann,
Karlsruhe Institute of Technology (KIT), Institute of Stochastics, Kaiserstra{\ss}e 89, D-76133 Karlsruhe, Germany.

E-mail: julia.hoerrmann@kit.edu

\medskip

\noindent
Daniel Hug,
Karlsruhe Institute of Technology (KIT), Institute of Stochastics, Kaiserstra{\ss}e 89, D-76133 Karlsruhe, Germany.

E-mail: daniel.hug@kit.edu

\medskip

\noindent
Michael Klatt,
Theoretische Physik, Friedrich-Alexander-Universit\"at Erlangen-N\"urnberg, Staudtstra{\ss}e 7B, D-91058 Erlangen, Germany.

E-mail: michael.klatt@fau.de

\medskip

\noindent
Klaus Mecke,
Theoretische Physik, Friedrich-Alexander-Universit\"at Erlangen-N\"urnberg, Staudtstra{\ss}e 7B, D-91058 Erlangen, Germany.

E-mail: klaus.mecke@physik.uni-erlangen.de

\end{document}